\documentclass[11pt,reqno]{amsart}
\usepackage{enumerate}
\usepackage{url}
\usepackage{cite}
\usepackage{comment}
\usepackage{graphicx}
\usepackage{upgreek}
\usepackage{amsrefs}
\usepackage{tikz}
\usetikzlibrary{arrows,shapes,trees,backgrounds}
\usepackage{fancyhdr}
\usepackage{amsfonts, amsmath, amssymb, amscd, amsthm, bm, cancel}
\usepackage[
linktocpage=true,colorlinks,citecolor=magenta,linkcolor=blue,urlcolor=magenta]{hyperref}
\usepackage[margin=1.1in]{geometry}
\usepackage{cases}
\newtheorem{prop}{Proposition}[section]
\newtheorem*{thmA}{Theorem A}
\newtheorem*{thmB}{Theorem B}
\newtheorem*{thmC}{Theorem C}
\newtheorem*{thmD}{Theorem D}
\newtheorem*{thmE}{Theorem E}
\newtheorem*{thmF}{Proposition A}
\newtheorem{thm}[prop]{Theorem}
\newtheorem{lem}[prop]{Lemma}
\newtheorem{cor}[prop]{Corollary}
\newtheorem{defn}[prop]{Definition}

\newtheorem{assu}[prop]{Assumption}
\theoremstyle{definition}
\newtheorem*{ack}{Acknowledgments}
\theoremstyle{remark}
\newtheorem{rem}[prop]{Remark}

\numberwithin{equation}{section}
\allowdisplaybreaks

\begin{document}
\begin{sloppypar}

\title[Minkowski type inequality]{A Minkowski type inequality in warped cylinders}
\author[S. Pan]{Shujing Pan}
\address{School of Mathematical Sciences, University of Science and Technology of China, Hefei 230026, P.R. China}
\email{\href{mailto:psj@ustc.edu.cn}{psj@ustc.edu.cn}}
\author[B. Yang]{Bo Yang}
\address{Department of Mathematical Sciences, Tsinghua University, Beijing 100084, P.R. China}
\email{\href{mailto:ybo@tsinghua.edu.cn}{ybo@tsinghua.edu.cn}}

\subjclass[2020]{52A40; 53C21; 53C42}
\keywords{Minkowski type inequality, warped cylinders, weakly mean convex, inverse mean curvature flow, weak solution}

\begin{abstract}
We prove a Minkowski type inequality for weakly mean convex and star-shaped hypersurfaces in warped cylinders which are asymptotically flat or hyperbolic. In particular, we show that this sharp inequality holds for outward minimizing hypersurfaces in the Schwarzschild manifold or the hyperbolic space using the weak solution of the inverse mean curvature flow. 
\end{abstract}

\maketitle

\tableofcontents

\section{Introduction}
The classical Minkowski inequality for a closed, convex hypersurface $\Sigma$ in $\mathbb{R}^{n+1}$ states as
\begin{equation}\label{MinTy}
	\int_{\Sigma}H\,d\mu\geq n\omega_n^{\frac{1}{n}}|\Sigma|^{\frac{n-1}{n}},
\end{equation}
where $H$ is the mean curvature of $\Sigma$, $\omega_n$ is the area of $n$-dimensional Euclidean unit sphere, and $|\Sigma|$ is the area of $\Sigma$. The equality holds in \eqref{MinTy} if and only if $\Sigma$ is a round sphere. Guan and Li \cite{G-L2009} showed that the inequality \eqref{MinTy} holds for weakly mean convex and star-shaped hypersurfaces using the inverse mean curvature flow in $\mathbb{R}^{n+1}$ introduced in \cites{Ger1990,Urbas1991}. In this paper, we say a hypersurface $\Sigma$ is weakly mean convex if the mean curvature of $\Sigma$ satisfies $H\geq 0$, and strictly mean convex if $H>0$. The Minkowski inequality \eqref{MinTy} has been generalized to the $h$-convex hypersurfaces in the hyperbolic space $\mathbb{H}^{n+1}$ by Ge-Wang-Wu \cite{GWW14} and convex hypersurfaces in the sphere $\mathbb{S}^{n+1}$ by Makowski-Scheuer \cite{MS-16}. Brendle, Guan and Li \cite{BGL} proved the following sharp inequality for strictly mean convex and star-shaped hypersurfaces in $\mathbb{H}^{n+1}$ using the inverse mean curvature flow:
\begin{align}\label{MinTy-2}
    \int_{\Sigma}{H}\,d\mu-n\mathrm{Vol}(\Omega)\geq\psi(|\Sigma|),
\end{align}
where $\psi$ is the  unique monotonically increasing function that gives equality on geodesic spheres, and $\Omega$ is the domain bounded by $\Sigma$. The quantity on the left side of the inequality \eqref{MinTy-2} is also known as a quermassintegral of $\Omega$. 

In this paper, we extend the Minkowski type inequality to the case of hypersurfaces in more general rotationally symmetric spaces.  We consider a warped cylinder $M^{n+1}=[a,+\infty)\times\mathbb{S}^n$ endowed with the metric
\begin{align}\label{metric-1}
 \bar{g}=dr^2+\varphi^2(r)g_{\mathbb{S}^n},
\end{align}
where $a\geq 0$, $\varphi\in C^{\infty}([a,+\infty),\mathbb{R}_+)$ is the warped function and $g_{\mathbb{S}^n}$ is the standard metric of the sphere $\mathbb{S}^n$. In particular, if we assume that $\varphi\in C^{\infty}([0,+\infty))$ satisfying
\begin{equation}\label{metric-D}
  \varphi'(r)=\sqrt{1-m{\varphi}^{1-n}+\kappa {\varphi}^2},\quad \varphi(0)=s_0,\quad \text{and} \quad \varphi'(0)=0,
\end{equation}
where $m>0$ and $\kappa\geq 0$ are two fixed real numbers, $s_0$ is the unique positive solution of the equation $1-m{s_0}^{1-n}+\kappa s_0^2=0$. Then, $M$ is known as the deSitter-Schwarzschild manifold. If $\kappa=0$, $M$ is the ordinary Schwarzschild manifold, while $\kappa=1$ corresponds to the anti-de Sitter-Schwarzschild (AdS-Schwarzschild) manifold.

In \cite{BHW16},  Brendle, Hung and Wang proved the following Minkowski type inequality for strictly mean convex and star-shaped hypersurfaces in the AdS-Schwarzschild manifold:
\begin{thmA}[\cite{BHW16}]\label{ThmA}
	Let $\Sigma$ be a compact, strictly mean convex and star-shaped hypersurface in the AdS-Schwarzschild manifold $M^{n+1}$, and  $\Omega$ be the region bounded by $\Sigma$ and the horizon $\partial M=\{0\}\times\mathbb{S}^n$. Then 
	\begin{equation}\label{Min-B}
		\int_{\Sigma}{\varphi'H}\,d\mu-n(n+1)\int_{\Omega}{\varphi'}\,dv\geq n\omega_n^{\frac{1}{n}}\big(|\Sigma|^{\frac{n-1}{n}}-|\partial M|^{\frac{n-1}{n}}\big).
	\end{equation}
Moreover, equality holds in \eqref{Min-B} if and only if $\Sigma$ is a radial coordinate sphere. 
\end{thmA}
For the purpose, they investigated the long time existence and asymptotic behavior of the inverse mean curvature flow in the AdS-Schwarzschild manifold. They found that the quantity
\begin{equation}
	Q(t)=|\Sigma_t|^{-\frac{n-1}{n}}\left(\int_{\Sigma_t}{\varphi'H}\,d\mu-n(n+1)\int_{\Omega_t}{\varphi'}\,dv+n\omega_n^{\frac{1}{n}}|\partial M|^{\frac{n-1}{n}}\right)
\end{equation}
is monotone decreasing along the inverse mean curvature flow and satisfies $\liminf_{t\to\infty}Q(t)\geq n\omega_n^{\frac{1}{n}}$, then inequality \eqref{Min-B} follows immediately. As limit cases, they also obtained the Minkowski type inequality in the Schwarzschild manifold and the hyperbolic space. Inspired by their work, inequality \eqref{Min-B} has been established in many other warped products, such as the Kottler space \cite{GWWX-15} and the Reissner-Nordström-anti-deSitter manifold \cites{CLZ19,Wang15}.

There are also Minkowski type inequalities different from \eqref{Min-B}. In \cite{Scheuer22}, Scheuer investigated a locally constrained inverse curvature flow in general warped products which satisfy mild assumptions and proved a new Minkowski type inequality for strictly convex surfaces, we summarize his result in the following theorem.
\begin{thmB}[\cite{Scheuer22}]
	Let $M=(a,b)\times\mathcal{S}_0$ be a 3-dimensional warped product space equipped with the metric $\bar{g}=dr^2+{\vartheta}^2(r)\sigma$, where $(\mathcal{S}_0,\sigma)$ is a compact Riemannian manifold of dimension 2 and $\vartheta\in C^{\infty}([a,b])$. Assume further that $(M,\bar{g})$ satisfies the following assumptions:
	\begin{itemize}
		\item[(i)] $\vartheta'>0$,
		\item[(ii)] either one of the following conditions holds:
		\begin{itemize}
			\item[(a)] $\vartheta''\geq 0$,
			\item[(b)] $\vartheta''\leq 0$ and 
			\begin{equation*}
				\partial_r\left(\frac{\vartheta''}{\vartheta}\right)\leq 0,
			\end{equation*}
		\end{itemize}
	    \item[(iii)] $\widehat{Ric}\geq (\vartheta'^2(r)-\vartheta''\vartheta(r))\sigma$ for all r, where $\widehat{Ric}$ is the Ricci curvature of the metric $\sigma$.
	\end{itemize}
Then, if $\Sigma\subset M$ is a strictly convex graph over $\mathcal{S}_0$ which encloses a domain $\Omega$, there holds
\begin{equation}\label{MinS}
	\int_{\Sigma}{H}\,d\mu+\int_{\Omega}{\overline{Ric}(\partial_r,\partial_r)}\,dv\geq\chi(|\Sigma|),
\end{equation}
where $\overline{Ric}$ denotes the Ricci curvature of $M$, $\chi$ is the function that gives equality on radial coordinate spheres. If equality holds, then $\Sigma$ is totally umbilical. If the associated quadratic forms in assumption (iii) satisfy the strict inequality on nonzero vectors, then inequality holds in \eqref{MinS} precisely on radial coordinate spheres.
\end{thmB} 

The functional 
\begin{equation*}
	\int_{\Sigma}{H}\,d\mu+\int_{\Omega}{\overline{Ric}(\partial_r,\partial_r)}dv
\end{equation*}
can be seen as a natural quermassintegral defined for smooth domains in general warped products, since it reduces to $\int_{\Sigma}{H}\,d\mu+nc\mathrm{Vol}(\Omega)$ in the space form $M^{n+1}(c)$, which coincides with the one given in integral geometry \cites{San04,Sol06}. Since the long time existence and convergence result is not available for the higher dimensional locally constrained curvature flow which Scheuer \cite{Scheuer22} considered, he could only prove \eqref{MinS} for strictly convex surfaces in $M^3$.

It remains an open question that whether inequality \eqref{MinS} holds for strictly (weakly) mean convex and star-shaped hypersurfaces in general warped products in addition to $\mathbb{R}^{n+1}$ and $\mathbb{H}^{n+1}$. Our paper gives a confirmation to this question in warped cylinders which are asymptotically flat or hyperbolic in the sense of Assumption \ref{ass-1}. 

\begin{assu}\label{ass-1}
    Let $n\geq 2$,  $M^{n+1}=[a,\infty)\times\mathbb{S}^n$ be a warped cylinder endowed with the metric \eqref{metric-1}. Fixing $\kappa\geq 0$, we suppose that $\varphi$ satisfies the following conditions for all $r\in(a,\infty)$:
    \begin{itemize}
        \item[(i)] $\varphi''(r)\geq\kappa \varphi(r)$,
        \item[(ii)] $0<\varphi'(r)\leq \sqrt{1+\kappa\varphi^2(r)}$,
        \item[(iii)] $\partial_r\left(\frac{\varphi''}{\varphi}\right)\leq 0$.
    \end{itemize}
    We will call $\{a\}\times\mathbb{S}^n$ the horizon of $M$, and denote it as $\partial M$.
\end{assu}
\begin{rem}\label{Rem-D}
    Besides the space form of constant non-positive curvature, the deSitter-Schwarzschild manifold given in \eqref{metric-D} with $\kappa\geq 0$
    also satisfies the conditions in Assumption \ref{ass-1}.
\end{rem}

Now we state our main theorem of the Minkowski type inequality for weakly mean convex and star-shaped hypersurfaces in warped cylinders $M^{n+1}$ which are asymptotically flat or hyperbolic in the sense of Assumption \ref{ass-1}. We say a hypersurface $\Sigma\subset M$ is star-shaped if its support function $u:=\bar{g}(\varphi\partial_r,\nu)>0$, where $\nu$ is the unit outward normal of $\Sigma$.

\begin{thm}\label{Min-In1}
	Let $(M^{n+1},\bar{g})$ satisfies Assumption \ref{ass-1}. Assume that $\Sigma$ is a compact, weakly mean convex and star-shaped hypersurface in $M$, and let $\Omega$ be the region bounded by $\Sigma$ and the horizon $\partial{M}$. Then
	\begin{equation}\label{Mintype}
		\int_{\Sigma}{H}\,d\mu+\int_{\Omega}{\overline{Ric}(\partial_r,\partial_r)}\,dv\geq \xi(|\Sigma|),
	\end{equation}
where $\xi$ is the unique monotonically increasing function that gives equality on radial coordinate spheres. Moreover, the equality holds in \eqref{Mintype} if and only if one of the following two cases holds:
\begin{itemize}
    \item[a)]$\Sigma$ is a radial coordinate sphere in $M$;
    \item[b)]there exists a radial coordinate ball $B(R)$(may be empty) such that $M\setminus B(R)$ is isometric to a space form of constant nonpositive curvature and  $\Sigma$ is a geodesic sphere in it.
\end{itemize}
\end{thm}

Combining with Remark \ref{Rem-D}, we have the following Minkowski type inequality for weakly mean convex and star-shaped hypersurfaces in the deSitter-Schwarzschild manifold which has nonpositive radial Ricci curvature.

\begin{cor}\label{Min-InD}
	Let $\Sigma$ be a closed, weakly mean convex and star-shaped hypersurface in the deSitter-Schwarzschild manifold $M^{n+1}$ with $\kappa\geq 0$ (see \eqref{metric-D} for definition), and let $\Omega$ be the region bounded by $\Sigma$ and the horizon $\partial{M}$. Then
	\begin{equation}\label{MintypeD}
		\int_{\Sigma}{H}\,d\mu+\int_{\Omega}{\overline{Ric}(\partial_r,\partial_r)}\,dv\geq \xi(|\Sigma|),
	\end{equation}
where $\xi$ is the unique monotonically increasing function that gives equality on radial coordinate spheres, and the equality holds in \eqref{MintypeD} if and only if $\Sigma$ is a radial coordinate sphere.
\end{cor}

In the second part of our paper, we focus on the Schwarzschild manifold $M^{n+1}$ and the hyperbolic space $\mathbb{H}^{n+1}$, and aim to prove the inequality \eqref{Mintype} for outward minimizing hypersurfaces using the weak solution of the inverse mean curvature flow. 

Let $\Omega$ be a bounded domain with smooth boundary $\partial\Omega$ in the Schwarzschild manifold $M^{n+1}$, then there are two cases: (i)\,$\Omega$ has only one boundary component $\Sigma=\partial\Omega$, and we say $\Sigma$ is null-homologous; (ii)\, $\Omega$ has two boundary components $\partial\Omega=\Sigma\cup\partial M$, and we say $\Sigma$ is homologous to the horizon. In both cases, the boundary hypersurface $\Sigma$ is said to be outward minimizing if whenever $\Omega^{\#}$ is a domain containing $\Omega$, then $|\partial\Omega^{\#}|\geq |\partial\Omega|$. From the first variational formula for area functional, an outward minimizing hypersurface must be weakly mean convex.

The theory of weak solution of the inverse mean curvature flow was developed by Huisken and Ilmanen \cite{HI-2001}, and was applied to prove the Riemannian Penrose inequality for asymptotically flat 3-manifold with nonnegative scalar curvature and other interesting problems. In particular, Huisken removed the star-shaped assumption and verified that the Minkowski inequality \eqref{MinTy} holds for outward-minimizing hypersurfaces in $\mathbb{R}^{n+1}$. Later, Wei \cite{Wei18} generalized Huisken's result to outward minimizing hypersurfaces in the Schwarzschild manifold, which we give a brief review here.

\begin{thmC}[\cite{Wei18}]\label{Thm-wei18}
	Let $\Omega$ be a bounded domain with smooth and outward minimizing boundary in the Schwarzschild manifold $(M^{n+1},\bar{g})$. Assume either
	\begin{itemize}
		\item[(i)] $n<7$, or 
		\item[(ii)] $n\geq 7$ and $\Sigma=\partial\Omega\setminus{\partial M}$ is homologous to the horizon.
	\end{itemize}
Then
\begin{equation}\label{Weak-Min}
	\frac{1}{n\omega_n}\int_{\Sigma}{fH}\,d\mu\geq \left(\frac{|\Sigma|}{\omega_n}\right)^{\frac{n-1}{n}}-m.
\end{equation}
\end{thmC}
\begin{rem}
	Our definition of the metric \eqref{metric-D} for the Schwarzschild manifold differs from the one in \cite{Wei18} by changing the variable $m$ to $2m$, and hence the inequality \eqref{Weak-Min} we expressed here is slightly different from the one in the original paper \cite{Wei18}.
\end{rem}

Later, McCormick \cite{Mc-18} proved inequality \eqref{Weak-Min} in a class of static asymptotically flat manifolds for $2\leq n<7$. Harvie and Wang \cite{HW-2024} extended McCormick's result to all 
dimension $n$ assuming the boundary is connected and characterized the equality case. Following the similar idea as in \cite{Wei18}, we would prove that inequality \eqref{Mintype} also holds for outward minimizing hypersurfaces $\Sigma$ in the Schwarzschild manifold. 

\begin{thm}\label{Min-In2}
Let $\Omega$ be a bounded domain with smooth and outward minimizing boundary in the Schwarzschild manifold $(M^{n+1},\bar{g})$. Assume either 
 \begin{itemize}
     \item[(i)] $n<7$, $\Sigma$ is null homologous and the volume of $\Omega$ is sufficiently large, or
     \item[(ii)] $\Sigma=\partial\Omega\setminus\partial M$ is homologous to the horizon.
 \end{itemize}
 Then
	\begin{equation}\label{Mintype2}
		\int_{\Sigma}{H}\,d\mu+\int_{\Omega}{\overline{Ric}(\partial_r,\partial_r)}\,dv\geq \xi(|\Sigma|),
	\end{equation}
where $\xi$ is the unique monotonically increasing function that gives equality on radial coordinate spheres, and the equality holds in \eqref{Mintype2} if and only if $\Sigma$ is a radial coordinate sphere.
\end{thm}
\begin{rem}
As we will show in \S\ref{Sec-4} and \S\ref{Sec-5}, the isoperimetric inequality plays an important role in our proof. According to Theorem D (see \S\ref{Sec-3}), the radial coordinate sphere is the unique solution to the isoperimetric problem among hypersurfaces which are homologous to the horizon in the Schwarzschild manifold, while the null homologous case is not discussed. In fact, there exist other isoperimetric regions when the volume is small enough (see \cite{BE13}). Following the same argument as in \cite{BE13}, we prove in \S\ref{Subsec-5.2} that the isoperimetric region is unique if the volume is sufficiently large. Therefore, the assumption that the volume of $\Omega$ is sufficiently large is necessary in the null homologous case.
\end{rem}
In the case of negative curvature, the weak solution of the inverse mean curvature flow with an outward minimizing initial data still exists in an asymptotically hyperbolic manifold. However, the blow-down argument fails at this time, which is very effective in the asymptotically flat case. Recently, Harvie \cite{H-2024} investigated the regularity of the weak solution of the inverse mean curvature flow in the hyperbolic space $\mathbb{H}^{n+1}$ with $2\leq n<7$. He showed that each slice of the flow is star-shaped after a long time, and then smooth. There are also other works concerning the regularity of the weak solution of the inverse mean curvature flow, such as \cites{HI-2008,LW17,Ne10,SZ-21}.
Combining the regularity results in \cites{H-2024} with the existence theory established in \cite{HI-2001},
we can also prove the Minkowski inequality for outward minimizing hypersurfaces in the hyperbolic space $\mathbb{H}^{n+1}$ with $2\leq n<7$.

\begin{thm}\label{Thm-Ads}
Let $\Omega$ be a bounded domain with smooth and outward minimizing boundary in the $(n+1)$-dimensional hyperbolic space $\mathbb{H}^{n+1}$ with $2\leq n<7$, then
\begin{equation}\label{Mintype4}
		\int_{\Sigma}{H}\,d\mu-n\mathrm{Vol}(\Omega)\geq \xi(|\Sigma|),
	\end{equation}
 where $\xi$ is the unique monotonically increasing function that gives equality on geodesic spheres, and the equality holds in \eqref{Mintype4} if and only if $\Sigma$ is a geodesic sphere. 
\end{thm}

The paper is organized as follows: In \S\ref{Sec-2}, we collect some preliminaries on the inverse mean curvature flow (IMCF), including the long time existence result and asymptotic behavior of the smooth solution to the IMCF, as well as the existence, uniqueness, compactness and regularity properties of the weak formulation of the IMCF. In \S\ref{Sec-3}, we use the isoperimetric inequality proved by Chodosh \cite{Cho15} to show an inequality between the weighted volume and the area for a closed hypersurface in a warped cylinder which satisfies Assumption \ref{ass-1}. In \S\ref{Sec-4}, we show that the geometric quantity $\mathcal{G}(t)$ defined in \eqref{defn-G} is monotone decreasing along the IMCF and satisfies $\lim_{t\to\infty}{\mathcal{G}(t)}\geq 0$, and then complete the proof of Theorem \ref{Min-In1} and Corollary \ref{Min-InD}. We also discuss the generalization of Theorem \ref{Min-In1} to Riemannian warped products, which need not be rotationally symmetric. In \S\ref{Sec-5}, we complete the proof of Theorem \ref{Min-In2} and Theorem \ref{Thm-Ads} using the weak solution of the IMCF.
 
\begin{ack}
	The research was supported by National Key R and D Program of China 2021YFA1001800 and 2020YFA0713100, China Postdoctoral Science Foundation No.2022M723057 and No.2024M751605, and Shuimu Tsinghua Scholar Program (No. 2023SM102). The authors would like to thank Professor Yong Wei for his helpful suggestions.
\end{ack}  

\section{The solution of the inverse mean curvature flow}\label{Sec-2}
In this section, we always assume that $M^{n+1}=[a,\infty)\times\mathbb{S}^n(a\geq 0)$ be a warped cylinder equipped with the metric 
\begin{equation}\label{metric-2}
    \bar{g}=dr^2+\varphi^2(r)g_{\mathbb{S}^n},
\end{equation}
where $\varphi'>0,\varphi''\geq 0$ for all $r\in(a,\infty)$. Consider a  family of embedding $X:\Sigma\times[0,T)\to M$ satisfying
\begin{align}\label{flow}
\partial_tX=\frac{1}{H}\nu,
\end{align}
where $H$ is the mean curvature of $\Sigma_t$ and $\nu$ is the unit outward normal. It has been shown in \cites{Scheuer17,Scheuer19} that the smooth solution of the flow \eqref{flow} starting from a strictly mean convex and star-shaped hypersurface remains to be strictly mean convex and star-shaped, and exists for all time $t\in[0,+\infty)$. In general, without some special assumption on the initial hypersurface, the singularities of the flow \eqref{flow} may occur. In \cite{HI-2001}, Huisken and Ilmanen used the level set approach and developed the weak solution of IMCF to overcome this problem.
 
\subsection{The smooth solution of IMCF}
In this subsection, we suppose that the initial hypersurface $\Sigma$ is strictly mean convex and star-shaped. In this setting, each flow hypersurface $\Sigma_t$ can be represented as a graph
\begin{align*}
    \Sigma_t=\{(r(\theta,t),\theta):\theta\in\mathbb{S}^n\},
\end{align*}
where $r(\theta,t)$ is a smooth function on $\mathbb{S}^n\times[0,+\infty)$. We define a function $\phi$ on $\mathbb{S}^n$ by 
\begin{align}\label{def-phi}
    \phi(\theta)=\Phi(r(\theta)),
\end{align}
where $\Phi(r)$ is a positive function satisfying $\Phi'(r)=\frac{1}{\varphi(r)}$. Define 
\begin{align*}
    \upsilon=\sqrt{1+\vert D\phi\vert^2_{\mathbb{S}^n}},
\end{align*}
where $D$ denotes the Levi-Civita connection on $\mathbb{S}^n$. Then the flow \eqref{flow} is equivalent to the following parabolic scalar equation for $r(\theta,t)$:
\begin{align}\label{flow-2}
\frac{\partial r}{\partial t}(\theta,t)=\frac{\upsilon}{H}.
\end{align}

Scheuer \cite{Scheuer17} dealt with the IMCF in warped cylinders of nonpositive radial curvature and generalized the results to Riemannian warped products in \cite{Scheuer19}, which means the base manifold may not be $\mathbb{S}^n$. Compared the asymptotic behaviors of the smooth solution to the IMCF in \cites{Scheuer17} and \cite{Scheuer19}, we find the latter one is more suitable for our use. In the following, we summarize the main results of \cite{Scheuer19} in the setting of warped cylinders.

\begin{thm}[\cite{Scheuer19}]\label{Scheuer-es}
     Let $(M^{n+1},\bar{g})$ be a warped cylinder given in \eqref{metric-2}. Assume that the warped function $\varphi$ satisfies the following assumptions:
     \begin{itemize}
     \item[(i)]
     \begin{equation}\label{SA-1}
          \limsup_{r\to\infty}\frac{\varphi''\varphi}{(\varphi')^2}<\infty\,\,\,\text{and}\,\,\, \limsup_{r\to\infty,\varphi''>0}\frac{\varphi'''\varphi}{\varphi'\varphi''}<\infty.
     \end{equation}
     \item[(ii)] In case that $\sup_{r>0}{\varphi'(r)}=\infty$, we assume 
     \begin{equation}\label{SA-2}
         \liminf_{r\to\infty}\frac{\varphi''\varphi}{(\varphi')^2}>0.
     \end{equation}
     \end{itemize}
    Let $\Sigma$ be a strictly mean convex and star-shaped hypersurface. Then the IMCF \eqref{flow} starting from $\Sigma$ has a unique smooth solution $\Sigma_t$ for all time $t\in[0,\infty)$. $\Sigma_t$ remains to be strictly mean convex, star-shaped and becomes umbilical with the rate 
     \begin{equation}\label{S-es}
         \left\vert{h_j^i-\frac{\varphi'}{\varphi}\delta_j^i}\right\vert\leq \frac{ct}{\varphi\varphi'},
     \end{equation}
     where the $t$-factor can be replaced by $e^{-\alpha t}$ for some positive $\alpha$ if $\varphi'$ is bounded.
\end{thm}

\subsection{The weak solution of IMCF}\label{sub-2.2} In this subsection, we recall the definition and some properties of the weak solution of the inverse mean curvature flow, which was firstly developed by Huisken and Ilmanen \cite{HI-2001} in general complete Riemannian manifolds with only a mild assumption at the infinity (e.g. asymptotically conic), and can be applied to the Schwarzschild manifold or the hyperbolic space directly. 

Assume that the flow \eqref{flow} is given by the level sets of a function $u:M\to\mathbb{R}$ via 
\begin{equation*}
\Omega_t:=\{x:u(x)<t\},\qquad \Sigma_t:=\partial {\Omega}_t.
\end{equation*}
Whenever $u$ is smooth with $\nabla u\neq 0$, equation \eqref{flow} is equivalent to 
\begin{equation}\label{eq-div}
\mathrm{div}_{M}\left(\frac{\nabla u}{|\nabla u|}\right)=|\nabla u|,
\end{equation}
where the left hand side gives the mean curvature of $\{u=t\}$ and the right hand side gives the inverse speed. 

Freezing the $|\nabla u|$ term on the right hand side, consider equation \eqref{eq-div} as the Euler-Lagrange equation of the functional
\begin{equation}
    J_u(v)=J_u^K(v):=\int_{K}{|\nabla v|+v|\nabla u|}\,dx.
\end{equation}

Let $\Omega$ be an open set with a boundary that is at least $C^1$. We say that $u$ is a weak solution of \eqref{eq-div} with initial condition $\Omega$ if

(1) $u\in C^{0,1}_{loc}(M)$ and $\Omega=\{u<0\}$;

(2) For every locally Lipschitz function $v$ such that $\{v\neq u\}\subset\subset\Omega$, there holds
    \begin{equation*}
     J_u^K(u)\leq J_u^K(v),
    \end{equation*}
    where the integration is performed over any compact set $K$ containing $\{u\ne v\}$.

Next, we introduce the notion of minimizing hulls. Let $A$ be an open set. We call $E$ is a minimizing hull (in $A$), if $E$ minimizes area on the outside in $A$, that is, if
\begin{equation*}
|\partial^{*}E\cap K|\leq |\partial^{*}F\cap K|
\end{equation*}
for any $F$ containing $E$ such that $F\setminus E\subset\subset A$, and any compact set $K$ containing $F\setminus E$. Here $\partial^{*}F$ denotes the reduced boundary of a set $F$ of locally finite perimeter. We say that $E$ is a strictly minimizing hull (in $A$) if equality holds implies that $F\cap A=E\cap A$ a.e.. The intersection of a countable collection (strictly) minimizing hulls is a (strictly) minimizing hull. Now let $E$ be any measurable set. Define $E'=E'_{A}$ to be the intersection of (the Lebesgue points of) all the strictly minimizing hulls that contain $E$. Working modulo sets of measure zero, this may be realized by a countable intersection, so $E'$ itself is a strictly minimizing hull, and open. We call $E'$ the strictly minimizing hull of $E$ (in $A$). Note that $E''=E'$.

We summarize the existence, uniqueness, compactness and regularity properties of the weak solution in the following theorem for later use.
\begin{thm}{\cite{HI-2001}}\label{Prop-Weak}
Let $\Omega$ be a bounded domain with smooth boundary in the Schwarzschild manifold $M^{n+1}$ or the hyperbolic space $\mathbb{H}^{n+1}$ with $n<7$. In the Schwarzschild case, we denote $\Sigma=\partial\Omega\setminus\partial M$ and if $\Sigma$ is null homologous, we fill-in the region $W$ bounded by the horizon $\partial M$ as in \cite[\S 6]{HI-2001}. In the hyperbolic case, we denote $\Sigma=\partial\Omega$. Then there exists a proper, locally Lipschitz solution $u$ of \eqref{eq-div} with initial condition $\Omega$ such that:\\
(1) For $t\geq 0$, $\Omega_{t}'=\mathrm{int}\{u\leq t\}$. Moreover, $\Omega_t=\{u<t\}$ is a minimizing hull in $M$ for $t>0$.  \\
(2) If we define $\Sigma_t:=\partial\Omega_t$ and $\Sigma_t':=\partial\Omega_{t}'=\partial\{u>t\}$, then $\Sigma_t$ and $\Sigma_t'$ both define increasing families of $C^{1,\alpha}$ hypersurfaces in $M$ which possess locally uniform $C^{1,\alpha}$ estimates depending only on the local Lipschitz bounds for $u$.\\
(3)For all $t>0$, we have 
\begin{equation}\label{C-1beta-con}
{\Sigma}_s\to {\Sigma}_t \quad \text{as}\,\, s \nearrow t,\qquad {\Sigma}_s\to {\Sigma}_t' \quad \text{as}\,\, s \searrow t
\end{equation}
in the sense of local $C^{1,\beta}$ convergence, $0<\beta<\alpha$. The second convergence also holds for $t=0$.\\
(4) The weak mean curvature $H$ satisfies $H=|\nabla u|$ for $a.e. t>0$ and $a.e. x\in\Sigma_t$.\\
(5) For all $t>0$, we have $|\Sigma_t|=e^t|\Sigma'|$, and $|\Sigma_t|=e^t|\Sigma|$ if $\Omega$ is a minimizing hull. In particular, if $\Sigma$ is outward minimizing, then the second equality holds.

When $n\geq 7$, the regularity and the convergence results remain true away from a closed singular set $Z$ of dimension at most $n-7$ and disjoint from $\bar{\Omega}$.  
\end{thm}

\section{The isoperimetric inequality in warped cylinders}\label{Sec-3}
In this section, we mainly deal with the (weighted) isoperimetric inequality in warped cylinders, we firstly introduce the notion of isoperimetric region.
\begin{defn}
    Consider a warped cylinder $M=[a,\infty)\times\mathbb{S}^n$ with the metric $\bar{g}=dr^2+\varphi^2(r)g_{\mathbb{S}^n}$. Assume that $\Omega$ is a Borel set and contains the horizon. We say $\Omega$ is isoperimetric among sets containing the horizon if
    \begin{equation*}
        |\partial^{*}\Omega'|\geq |\partial^{*}\Omega|
    \end{equation*}
    for all Borel sets $\Omega'$ containing the horizon with $\mathrm{Vol}(\Omega')=\mathrm{Vol}(\Omega)$. Here $\partial^{*}\Omega$ denotes the reduced boundary of $\Omega$. Moreover, if equality holds only when $\Omega'=\Omega$ away from a set of measure zero, we say that $\Omega$ is uniquely isoperimetric among sets containing the horizon.
\end{defn}
\begin{rem}\label{rem-reduce}
    It's known that (see \cite[Remark 15.1-Remark 15.3]{Maggi-12}) for a set of locally finite perimeter $E$, up to modification on sets of measure zero, the reduced boundary $\partial^{*}E$ and the topological boundary $\partial E$ satisfy $\overline{\partial^{*}E}=\partial E$. In particular, if $E$ is an open set with $C^1$ boundary, then $\partial^{*}E=\partial E$. For another set $F$ of locally finite perimeter, if $\mathrm{Vol}(E\triangle F)=0$, then $\partial^{*}E=\partial^{*}F$.
\end{rem}
Bray \cite{Bray97} developed an isoperimetric comparison theorem to prove that radial coordinate 2-balls of the 3-dimensional Schwarzschild manifold are isoperimetric in their homology, and then Bray and Morgan \cite{BM2002} used the method to solve the isoperimetric problem in more general spherically symmetric manifolds. Later, Corvino, Gerek, Greenberg and Krummel \cite{CGGK0722} adapted Bray's techniques to show that radial coordinate balls in AdS-Schwarzschild are isoperimetric. In \cite{Cho15}, applying a similar manner, Chodosh proved the following theorem for warped cylinders:

\begin{thmD}{\cite[\S 3.2]{Cho15}}\label{Ch-isop}
	Consider a warped cylinder $M=[a,\infty)\times\mathbb{S}^n$ with the metric $dr^2+\varphi^2(r)g_{\mathbb{S}^n}$. Fixing $\kappa\geq0$, we suppose that 
	\begin{itemize}
		 \item[(i)] $\varphi''(r)\geq\kappa \varphi(r)$,
        \item[(ii)] $0\leq\varphi'(r)\leq \sqrt{1+\kappa\varphi^2(r)}$.
	\end{itemize}
 Then for all $r\geq a$, the radial coordinate balls $B(r)=[a,r)\times\mathbb{S}^n$ are isoperimetric among sets containing the horizon, uniquely if $\varphi'(r)< \sqrt{1+\kappa\varphi^2(r)}$.
\end{thmD}
\begin{rem}
    In his thesis \cite{Cho15}, Chodosh solved the isoperimetric problem in more general warped products. For the sake of simplicity, we reduce his results to warped cylinders in Theorem D.
\end{rem}

In order to establish the Minkowski inequality in the warped cylinders which are asymptotically flat or hyperbolic, we need the following weighted isoperimetric inequality, which is a direct consequence of the classical isoperimetric inequality.
\begin{thm}\label{thm-isop}
    Assume that $(M^{n+1},\bar{g})$ is a warped cylinder satisfying the same conditions as in Theorem D. Let $\eta(r)$ be a nonincreasing,  positive function, and $\Sigma\subset M$ be a closed, $C^1$ hypersurface which is homologous to the horizon $\partial M$, and $\Omega$ be the region bounded by $\Sigma$ and $\partial M$, 
    then 
    \begin{align}\label{isop-1}
        \int_{\Omega}\eta(r)dv\leq \xi_{\eta}(\vert\Sigma\vert),
    \end{align}
    where $\xi_{\eta}$ is the unique monotonically increasing function that gives equality on radial coordinate spheres.
    
    In particular, if $(M^{n+1},\bar{g})$ satisfies Assumption \ref{ass-1}, we have  
    \begin{align}\label{isop-2}
        \int_{\Omega}\frac{\varphi''}{\varphi}(r)dv\leq \xi_{1}(\vert\Sigma\vert),
    \end{align}
    where $\xi_{1}$ is the unique associated function. Moreover, if either the second inequality in Assumption \ref{ass-1} (ii) or the inequality in Assumption \ref{ass-1} (iii) is strict, then the equality holds in \eqref{isop-2} if and only if $\Sigma$ is a radial coordinate sphere.
\end{thm}
\begin{proof}
Since $\Sigma$ is $C^1$, then by Theorem D and Remark \ref{rem-reduce}, we know that the sharp isoperimetric inequality holds in $M$, that is 
\begin{align}\label{isop2-eq1}
    \text{Vol}(\Omega)\leq \xi_{0}(\vert\Sigma\vert),
\end{align}
where $\xi_0$ is the unique associated function that gives equality on radial coordinate spheres. Next, suppose that $B(r)$ is the domain bounded by $\partial M$ and the radial coordinate sphere $\{r\}\times\mathbb{S}^n$ such that $\text{Vol}(B(r))=\text{Vol}(\Omega$). Since $\eta$ is nonincreasing with respect to $r$, we have 
\begin{align}\label{isop2-eq3}
\int_{\Omega}\eta dv&=\int_{\Omega\cap B(r)}\eta dv+\int_{\Omega\setminus B(r)}\eta dv\nonumber\\
&\leq\int_{\Omega\cap B(r)}\eta dv+\int_{B(r)\setminus \Omega}\eta dv\\
&=\int_{B(r)}\eta dv,\nonumber
\end{align}
where in the inequality we used the fact that $\text{Vol}(\Omega\setminus B(r))=\text{Vol}(B(r)\setminus\Omega)$. It follows that the radial coordinate sphere $\{r\}\times \mathbb{S}^n$ maximizes the weighted volume among hypersurfaces enclosing fixed volume with the horizon $\partial M$, which means there exists an increasing, associated function $\widehat{\xi}_{\eta}$ such that 
\begin{align}\label{isop2-eq2}
\int_{\Omega}\eta dv\leq\widehat{\xi}_{\eta}(\text{Vol}(\Omega)),
\end{align}
where the equality holds if $\Sigma$ is a radial coordinate sphere.
Combining \eqref{isop2-eq1} with \eqref{isop2-eq2} gives \eqref{isop-1}. In particular, if $(M^{n+1},\bar{g})$ satisfies Assumption \ref{ass-1}, by the condition (iii), $\frac{\varphi''}{\varphi}$ is a nonincreasing function with respect to $r$, then the inequality \eqref{isop-2} follows from \eqref{isop-1} immediately.

Now, we deal with the equality case. If the second inequality in Assumption \ref{ass-1} (ii) is strict, then  the uniqueness in the isoperimetric Theorem D implies that the equality holds in \eqref{isop-2} if and only if $\Omega$ is a radial coordinate ball away from a set of measure zero. Since $\Omega$ is $C^1$, then by Remark \ref{rem-reduce}, $\Omega$ is precisely a radial coordinate ball and hence $\Sigma$ is a radial coordinate sphere. On the other hand, if the inequality in Assumption \ref{ass-1} (iii) is strict, then the function $\frac{\varphi''}{\varphi}$ is strictly decreasing. Hence, the inequality \eqref{isop2-eq3} is strict when $\Omega\neq B(r)$, this gives the uniqueness of the weighted isoperimetric inequality \eqref{isop-2} and completes the proof.
\end{proof}
\section{Proof of Theorem \ref{Min-In1} and Corollary \ref{Min-InD}}\label{Sec-4}
In this section, we will apply properties of the smooth solution of IMCF to prove the Minkowski inequality in the warped cylinder $M^{n+1}$ which satisfies Assumption \ref{ass-1}, and then discuss the generalization of Theorem \ref{Min-In1} to Riemannian warped products, which need not be rotationally symmetric.
\subsection{Asymptotic estimates and monotonicity along the IMCF}
Firstly, we show that if $M^{n+1}$ satisfies Assumption \ref{ass-1}, then it satisfies the conditions in Theorem \ref{Scheuer-es} and has the following asymptotic behavior estimates.
\begin{lem}\label{Lem-4.1}
    Let $(M^{n+1},\bar{g})$ be a warped cylinder satisfying Assumption \ref{ass-1}, then it satisfies the conditions in Theorem \ref{Scheuer-es}. Moreover, if $\Sigma_t$ are the solution hypersurfaces of the IMCF \eqref{flow} starting from a strictly mean convex and star-shaped hypersurface $\Sigma$, then we have the following asymptotic behavior estimates:
    \begin{itemize}
        \item[(a)] If $\kappa>0$, then the mean curvature $H$ of $\Sigma_t$ satisfies
        \begin{equation}\label{es-H}
            H=n\sqrt{\kappa}+O(te^{-\frac{2}{n}t}).
        \end{equation}
        \item[(b)] If $\kappa=0$, then there exist positive constants $\alpha, \bar{\varphi}$ and a positive function $\varepsilon(t)$ such that $\lim_{t\to\infty}\varepsilon(t)=0$ with
        \begin{align}
            &|\varphi H-n{\varphi}'|=O(e^{-\alpha t}),\label{es-phi H}\\
            &d\mu_t=\varphi^n\left(1+O(e^{-2\alpha t})\right),\label{es-gij}\\
            &\bar{\varphi}e^{\frac{t}{n}-\varepsilon(t)}\leq\varphi\leq\bar{\varphi}e^{\frac{t}{n}+\varepsilon(t)},\label{es-phitwo}
        \end{align}
        where $d\mu_t$ is the area element of $\Sigma_t$.
    \end{itemize}
\end{lem}
\begin{proof}
 (a) We first assume that $\kappa>0$. By Assumption \ref{ass-1}, $\varphi''\geq\kappa\varphi>0$ and hence $\varphi$ is a strictly increasing and convex function, which implies $\lim_{r\to\infty}\varphi(r)=\infty$. Define
 \begin{equation*}
     Q_1(r)=\frac{(\varphi')^2}{1+\kappa\varphi^2}.
 \end{equation*}
 Then by Assumption \ref{ass-1} (ii) and direct calculation, we see that $Q_1(r)\leq 1$ and $Q_1'(r)\geq 0$. Hence $\varphi'(r)$ is a strictly increasing function in $r$ with $\lim_{r\to\infty}{\varphi'(r)}=\infty$. By Assumption \ref{ass-1} (i) and (iii), the function $\frac{\varphi''}{\varphi}$ has a positive limit when $r$ tends to infinity. Then by L'Hospital's rule, we have
 \begin{equation*}
     1\geq \lim_{r\to\infty}Q_1(r)=\lim_{r\to\infty}\frac{\varphi''}{\kappa\varphi}\geq 1,
 \end{equation*}
 which yields 
 \begin{equation}\label{es-limit1}
     \lim_{r\to\infty}Q_1(r)=\lim_{r\to\infty}\frac{\varphi''}{\kappa\varphi}=1,
 \end{equation}
 from which we deduce 
 \begin{equation}\label{es-limit2}
     \lim_{r\to\infty}\frac{\varphi''\varphi}{(\varphi')^2}=\lim_{r\to\infty}\left[\frac{\varphi''}{\varphi}\frac{1+\kappa\varphi^2}{(\varphi')^2}\frac{\varphi^2}{1+\kappa\varphi^2}\right]=1.
 \end{equation}
 Moreover, by Assumption \ref{ass-1} (iii), we have
 \begin{equation}\label{es-de}
     \varphi'''\varphi-\varphi''\varphi'=\varphi^2\partial_r\left(\frac{\varphi''}{\varphi}\right)\leq 0,
 \end{equation}
 and hence 
 \begin{equation}\label{es-limit3}
     \limsup_{r\to\infty,\varphi''>0}\frac{\varphi'''\varphi}{\varphi'\varphi''}\leq 1.
 \end{equation}
 Indeed, by a simple argument using Lagrange mean value theorem, we could prove the superior limit is precisely 1. Combining \eqref{es-limit2} with \eqref{es-limit3}, we see that in case $\kappa>0$, $M^{n+1}$ satisfies the conditions in Theorem \ref{Scheuer-es}. Then, we turn to show the estimate \eqref{es-H}. Note that 
 \begin{equation*}
     1-Q_1(r)=\frac{1+\kappa\varphi^2-(\varphi')^2}{1+\kappa\varphi^2},
 \end{equation*}
 since $Q_2(r):=1+\kappa\varphi^2-(\varphi')^2\geq 0$ and $Q_2'(r)\leq 0$ by Assumption \ref{ass-1} (i) and (ii), we have $1-Q_1(r)=O(\varphi^{-2})$. This means
\begin{equation}
    \frac{\varphi'}{\sqrt{1+\kappa\varphi^2}}=1+O(\varphi^{-2}).
\end{equation}
Therefore,
\begin{equation}\label{es-limit4}
    \frac{\varphi'}{\varphi}=\frac{\varphi'}{\sqrt{1+\kappa\varphi^2}}\frac{\sqrt{1+\kappa\varphi^2}}{\varphi}=\sqrt{\kappa}+O(\varphi^{-2})\quad \text{and} \quad \varphi'=O(\varphi).
\end{equation}
It's known that along the IMCF \eqref{flow}, $\varphi=O(e^{\frac{t}{n}})$ (see \cites{Scheuer17,Scheuer19} for example). Combining this fact with \eqref{es-limit4} and \eqref{S-es} gives the estimate \eqref{es-H}.

(b) We next deal with the case $\kappa=0$. By Assumption \ref{ass-1}, there exists a constant $0<\bar{\varphi}'\leq1$, such that $\lim_{r\to\infty}{\varphi'(r)}=\bar{\varphi}'$. 
Define 
\begin{equation}\label{de-Q2}
    Q_3(r)=1-(\varphi')^2+\varphi\varphi'',
\end{equation}
then $Q_3(r)\geq 0$ by Assumption \ref{ass-1} and $Q_3'(r)\leq 0$ by \eqref{es-de}. Denote $Q_0:=\lim_{r\to\infty}{Q_2(r)}\geq 0$, then we have $\lim_{r\to\infty}{\varphi\varphi''}=Q_0+(\bar{\varphi}')^2-1$ and hence
\begin{equation}\label{es-case2-1}
    \lim_{r\to\infty}\frac{\varphi''\varphi}{(\varphi')^2}=\frac{Q_0+(\bar{\varphi}')^2-1}{(\bar{\varphi}')^2}.
\end{equation}
Combining \eqref{es-case2-1} with \eqref{es-limit3} implies that in case $\kappa=0$, $M^{n+1}$ satisfies the conditions in Theorem \ref{Scheuer-es}, the estimate \eqref{S-es} says that
\begin{equation*}
     \left\vert{h_j^i-\frac{\varphi'}{\varphi}\delta_j^i}\right\vert\leq \frac{ce^{-\alpha t}}{\varphi\varphi'},
\end{equation*}
which implies
\begin{equation}\label{es-case2-2}
    |\varphi H-n\varphi'|=O(e^{-\alpha t}).
\end{equation}
This is \eqref{es-phi H}. Next, the argument in \cite[Page 1140-1141]{Scheuer19} yields that
\begin{equation}\label{decay-g}
    |D\phi|\leq e^{-\alpha t},
\end{equation}
where the function $\phi$ is defined in \eqref{def-phi} and there exist a constant $\tilde{\varphi}$ and a positive function $\varepsilon(t)$ such that $\lim_{t\to\infty}{\varepsilon(t)}=0$ with
\begin{equation}\label{es-log}
    \tilde{\varphi}-\varepsilon(t)\leq\log{\varphi}-\frac{t}{n}\leq \tilde{\varphi}+\varepsilon(t).
\end{equation} 
Recall that the metric induced on $\Sigma_t$ has the expression
\begin{equation}\label{ex-metric}
    g_{ij}=\varphi^2(\sigma_{ij}+\phi_i\phi_j),
\end{equation}
where $\sigma_{ij}$ is the canonical metric of $\mathbb{S}^n$. Then combining \eqref{ex-metric} with \eqref{decay-g} gives \eqref{es-gij},
while \eqref{es-log} implies \eqref{es-phitwo} with $\bar{\varphi}=e^{\tilde{\varphi}}$. This completes the proof of Lemma \ref{Lem-4.1}.
\end{proof}
Let $\Sigma$ be a strictly mean convex and star-shaped hypersurface in $M$, we evolve it by the inverse mean curvature flow \eqref{flow}. Assume that $\Sigma_t$ are the flow hypersurfaces and $\Omega_t$ are the domains bounded by $\Sigma_t$ and $\partial M$. We define the quantities
\begin{align}\label{defn-W}
    \mathcal{W}(\Sigma_t)=\int_{\Sigma_t}{H}\,d\mu+\int_{\Omega_t}{\overline{Ric}(\partial_r,\partial_r)}\,dv,
\end{align}
and
\begin{align}\label{defn-G}
    \mathcal{G}(t)=\vert\Sigma_t\vert^{-\frac{n-1}{n}}\left(\mathcal{W}(\Sigma_t)-\xi(\vert\Sigma_t\vert)\right),
\end{align}
where $\xi$ is the associated function defined in Theorem \ref{Min-In1}.

Firstly, we recall the following evolution equations of the area element and the mean curvature for evolving hypersurfaces under a general flow for later use. The proof is standard and can be found in many references, see \cites{BHW16,LW17} for example.
\begin{prop}\label{Prop-general-evo}
Under the general flow
\begin{equation}\label{flow-general}
\partial_t x=\mathcal{F}\nu,
\end{equation}
for some normal speed function $\mathcal{F}$, the area element and the mean curvature of $\Sigma_t$ evolve by 
\begin{align}
    \partial_td\mu_t=&\mathcal{F}H d\mu_t,\label{eveq-area}\\
    \partial_tH=&-\Delta{\mathcal{F}}-\mathcal{F}\vert A\vert^2-\mathcal{F}\overline{Ric}(\nu,\nu), \label{eveq-H}
\end{align}
where $\Delta$ is the Laplacian operator with respect to the induced metric on $\Sigma_t$.
\end{prop}

Then, using the weighted isoperimetric inequality \eqref{isop-2}, we obtain
\begin{prop}\label{monprop}
   Assume that $(M^{n+1},\bar{g})$ is a warped cylinder satisfying Assumption \ref{ass-1}. Then,
   along the flow \eqref{flow}, we have $$\mathcal{G}'(t)\leq0.$$ Moreover,  the equality holds if and only if either one of the following two cases holds:
\begin{itemize}
    \item[a)]$\Sigma_t$ is a radial coordinate sphere in $M$;
    \item[b)]there exists a radial coordinate ball $B(R)$(may be empty) such that $M\setminus B(R)$ is isometric to a space form of constant nonpositive curvature and  $\Sigma_t$ is  a geodesic sphere in it.
\end{itemize}
\end{prop}
\begin{proof}   
    Combining \eqref{eveq-area}, \eqref{eveq-H} with $\mathcal{F}=1/H$ and using the coarea formula, we derive that
    \begin{align}\label{mon-eq2}
       \frac{d}{dt}\mathcal{W}(\Sigma_t)=& \frac{d}{dt}\left(\int_{\Sigma_t}{H}\,d\mu_t+\int_{\Omega_t}{\overline{Ric}(\partial_r,\partial_r)}\,dv\right)\nonumber\\
       =&\int_{\Sigma_t}\left[-\Delta\frac{1}{H}-\frac{\vert A\vert^2}{H}-\frac{\overline{Ric}(\nu,\nu)}{H}+H+\frac{\overline{Ric}(\partial_r,\partial_r)}{H}\right]\,d\mu_t\nonumber\\
        \leq&\int_{\Sigma_t}\left[\frac{n-1}{n}H+\frac{1}{H}(\overline{Ric}(\partial_r,\partial_r)-\overline{Ric}(\nu,\nu))\right]\,d\mu_t,
    \end{align}
    where we used the fact that $\vert A\vert^2\geq\frac{H^2}{n}$. 

     Recall that, by a direct calculation (see e.g \cites{GL15}), we have the Ricci curvature in $M$:
    \begin{align}\label{mon-eq1}
    \overline{Ric}=-n\frac{\varphi''}{\varphi}dr^2-[(n-1)((\varphi')^2-1)+\varphi\varphi'']g_{\mathbb{S}^n}.      
    \end{align}
   Hence, 
    \begin{align}\label{mon-eq3}
    \overline{Ric}(\partial_r,\partial_r)-\overline{Ric}(\nu,\nu)=&(n-1)\frac{(\varphi')^2-\varphi\varphi''-1}{\varphi^2}(1-\frac{u^2}{\varphi^2}),
   \end{align}
   where $u=\langle\varphi\partial_r,\nu\rangle$ is the support function of $\Sigma_t$. Since the warped function $\varphi$ of $M$ satisfies conditions (i) and (ii) in Assumption \ref{ass-1}, we have
    \begin{align}\label{mon-eq4}
      (\varphi')^2-\varphi\varphi''-1\leq 1+\kappa\varphi^2-\kappa\varphi^2-1
      =0.
    \end{align}
    Substituting \eqref{mon-eq4} into \eqref{mon-eq3}, we have
    \begin{align}\label{mon-eq11}
    \overline{Ric}(\partial_r,\partial_r)-\overline{Ric}(\nu,\nu)\leq0.
    \end{align}
    Combining this with \eqref{mon-eq2}, we get
    \begin{align}\label{mon-eq5}
     \frac{d}{dt}\mathcal{W}(\Sigma_t)\leq \frac{n-1}{n}\int_{\Sigma_t}{H}\,d\mu_t.
    \end{align}
    
    Next, let $S(r(t))=\{r(t)\}\times\mathbb{S}^n$ be a family of radial coordinate spheres evolving under the inverse mean curvature flow  in $M$, $B(r(t))$ be the domain bounded by $\partial M$ and $S(r(t))$. Then, the above calculation implies that
    \begin{align}\label{mon-eq6}
     \frac{d}{dt}\mathcal{W}(S(r(t)))=\frac{n-1}{n}\int_{S(r(t))}{H}\,d\mu_t.
    \end{align}
    By the definition of the function $\xi(x):\mathbb{R}_+\to\mathbb{R}_+$, we know that
    \begin{align*}
       \mathcal{W}(S(r(t)))= \int_{S(r(t))}{H}\,d\mu_t+\int_{B(r(t))}{\overline{Ric}(\partial_r,\partial_r)}\,dv=\xi(\vert S(r(t))\vert),\ \forall t\in[0,+\infty).
    \end{align*}
   Taking the derivative on both sides of the above equation, and using \eqref{mon-eq6} we get
    \begin{align}\label{mon-eq7}
       \xi'(\vert S(r(t))\vert)\vert S(r(t))\vert=&\frac{n-1}{n}\int_{S(r(t))}{H}\,d\mu_t\nonumber\\
       =&\frac{n-1}{n}\left(\xi(\vert S(r(t))\vert)-\int_{B(r(t))}{\overline{Ric}(\partial_r,\partial_r)}\,dv\right).
    \end{align}
    Note that $\overline{Ric}(\partial_r,\partial_r)=-n\frac{\varphi''}{\varphi}$, by the weighted isoperimetric inequality \eqref{isop-2} we built in Theorem \ref{thm-isop}, there holds
    \begin{align}\label{mon-eq8}
    -\int_{B(r(t))}{\overline{Ric}(\partial_r,\partial_r)}\,dv=&n\int_{B(r(t))}\frac{\varphi''}{\varphi}\,dv=n\xi_{1}(\vert S(r(t))\vert).
    \end{align}
    Combining \eqref{mon-eq8} with \eqref{mon-eq7}, we claim that the function $\xi(x)$ satisfies the following equation:
    \begin{align}\label{mon-eq9}
        \xi'(x)x=\frac{n-1}{n}(\xi(x)+n\xi_1(x)).
    \end{align}
Since the area of $\Sigma_t$ satisfies $\frac{d}{dt}\vert\Sigma_t\vert=\vert\Sigma_t\vert$, we now have
\begin{align}\label{mon-eq10}
    \frac{d}{dt}\xi(\vert\Sigma_t\vert)=&\xi'(\vert\Sigma_t\vert)\vert\Sigma_t\vert\nonumber\\=&\frac{n-1}{n}(\xi(\vert\Sigma_t\vert)+n\xi_1(\vert\Sigma_t\vert))\nonumber\\
    \geq&\frac{n-1}{n}\left(\xi(\vert\Sigma_t\vert)-\int_{\Omega_t}{\overline{Ric}(\partial_r,\partial_r)}\,dv\right),
\end{align}
where in the second equality we used \eqref{mon-eq9}, and in the last inequality we used the weighted isoperimetric inequality \eqref{isop-2}.
Hence, combining \eqref{mon-eq10} with \eqref{mon-eq5} gives
\begin{align*}
    \frac{d}{dt}\left(\mathcal{W}(\Sigma_t)-\xi(\vert\Sigma_t\vert)\right)
        \leq\frac{n-1}{n}\left(\mathcal{W}(\Sigma_t)-\xi(\vert\Sigma_t\vert)\right).
\end{align*}
Now, we conclude that $\mathcal{G}'(t)\leq0$ by using $\frac{d}{dt}\vert\Sigma_t\vert=\vert\Sigma_t\vert$ again. 

Finally, we treat the equality case.  $\mathcal{G}'(t)=0$ forces the equality to hold in \eqref{mon-eq2}, which is equivalent to $\Sigma_t$ being a totally umbilical hypersurface in $M$ and satisfying $\overline{Ric}(\partial_r,\partial_r)-\overline{Ric}(\nu,\nu)=0$. As mentioned in the proof of Lemma \ref{Lem-4.1}, the quantity 
\begin{equation*}
 Q_1(r)=\frac{(\varphi')^2}{1+\kappa\varphi^2}
\end{equation*}
is nondecreasing in $r$. Combining this fact with condition (ii) in Assumption \ref{ass-1} ($Q_1(r)\leq 1$), it's sufficient to consider the following two cases for $M$:

Case 1. $\varphi'<\sqrt{1+\kappa\varphi^2}$, $\forall r\in[a,+\infty)$.
In this case, by the condition (i) in Assumption \ref{ass-1}, we have $(\varphi')^2-\varphi\varphi''-1<0$. Combining this with \eqref{mon-eq3}, $\overline{Ric}(\partial_r,\partial_r)-\overline{Ric}(\nu,\nu)=0$ implies that $\Sigma_t$ is a radial coordinate sphere.

Case 2. $\exists R>a$ such that $\varphi'<\sqrt{1+\kappa\varphi^2}$ for $r\in[a,R)$ and $\varphi'=\sqrt{1+\kappa\varphi^2}$ for $r\in[R,+\infty)$, which implies that $M\setminus B(R)$ is isometric to a space form of constant nonpositive curvature. In this case, if $\min_{\Sigma_t}r<R$, by the similar argument as in Case 1, we have $\Sigma_t\cap B(R)$ is a part of a radial coordinate sphere. This implies $\Sigma_t\subset B(R)$ is a radial coordinate sphere. If $\min_{\Sigma_t}r\geq R$, then $\Sigma_t$ is a totally umbilical hypersurface in $M\setminus B(R)$, which means $\Sigma_t$ is a geodesic sphere. 

Reversely, if $\Sigma_t$ is a radial coordinate sphere, then $\mathcal{G}'(t)=0$ obviously. If $\Sigma_t\subset M\setminus B(R)$ is a geodesic sphere in a space form of nonpositive constant curvature, and $\Omega_t$ is the domain bounded by $\Sigma_t$ and $\partial M$, we turn to show that $\mathcal{G}'(t)=0$. Since $\overline{Ric}=-n\kappa\bar{g}$ in $M\setminus B(R)$, by \eqref{mon-eq2} we have 
\begin{align}\label{mon-eq12}
 \frac{d}{dt}\mathcal{W}(\Sigma_t)=\frac{n-1}{n}\int_{\Sigma_t}Hd\mu_t.   
\end{align}
 On the other hand, assume that $\widehat{\Sigma}_t$ is the coordinate sphere with $\vert\widehat{\Sigma}_t\vert=\vert\Sigma_t\vert$ and $\widehat{\Omega}_t$ is the radial ball bounded by $\partial M$ and $\widehat{\Sigma}_t$. Then
\begin{align}\label{mon-eq13}
    \int_{\Omega_t}{\overline{Ric}(\partial_r,\partial_r)}\,dv=&\int_{B(R)}{\overline{Ric}(\partial_r,\partial_r)}\,dv-n\kappa\text{Vol}_0(\Omega_t\setminus B(R))\nonumber\\
    =&\int_{B(R)}{\overline{Ric}(\partial_r,\partial_r)}\,dv-n\kappa\text{Vol}_0(\Omega_t)+n\kappa\text{Vol}_0(B(R))\nonumber\\
    =&\int_{B(R)}{\overline{Ric}(\partial_r,\partial_r)}\,dv-n\kappa\text{Vol}_0(\widehat{\Omega}_t\setminus B(R))=\int_{\widehat{\Omega}_t}{\overline{Ric}(\partial_r,\partial_r)}\,dv,
\end{align}
where $\text{Vol}_0$ is the volume in the space form. Hence, by \eqref{mon-eq9} we have 
\begin{align}\label{mon-eq14}
\frac{d}{dt}\xi(\vert\Sigma_t\vert)=&\frac{n-1}{n}(\xi(\vert\Sigma_t\vert)+n\xi_1(\vert\Sigma_t\vert))\nonumber\\
=&\frac{n-1}{n}(\xi(\vert\Sigma_t\vert)+n\xi_1(\vert\widehat{\Sigma}_t\vert))\nonumber\\
=&\frac{n-1}{n}\left(\xi(\vert\Sigma_t\vert)-\int_{\widehat{\Omega}_t}{\overline{Ric}(\partial_r,\partial_r)}\,dv\right)\nonumber\\
=&\frac{n-1}{n}\left(\xi(\vert\Sigma_t\vert)-\int_{\Omega_t}{\overline{Ric}(\partial_r,\partial_r)}\,dv\right),
\end{align}
where we used the weighted isoperimetric inequality \eqref{isop-2} and the identity \eqref{mon-eq13}. Now, combining \eqref{mon-eq12} with \eqref{mon-eq14} gives $\mathcal{G}'(t)=0$.
This completes the proof.
\end{proof}
\subsection{Proof of the Minkowski inequality} 
In this subsection, we complete the proof of Theorem \ref{Min-In1} and Corollary \ref{Min-InD}. Firstly, we investigate the limit of $\mathcal{G}(t)$ as $t\to\infty$.
\begin{prop}\label{limprop}
   Under the flow \eqref{flow}, we have 
   \begin{align*}
       \lim_{t\to\infty}\mathcal{G}(t)\geq0.
   \end{align*}
\end{prop}
\begin{proof}
 Let $\widehat{\Sigma}_t$ be a family of radial coordinate spheres evolving under the IMCF with $\vert\widehat{\Sigma}_0\vert=\vert\Sigma\vert$, and $\widehat{\Omega}_t$ be the domain bounded by $\widehat{\Sigma}_t$ and $\partial M$. Since $\frac{d}{dt}\vert\Sigma_t\vert=\vert\Sigma_t\vert$, we have
\begin{align}\label{mon2-eq1}
\vert\widehat{\Sigma}_t\vert=\vert\Sigma_t\vert=e^t\vert\Sigma\vert.
\end{align}

First, we consider the case $\kappa>0$. Recall the asymptotic estimate \eqref{es-H} of the mean curvature:
\begin{align}\label{mon2-eq2}
  H=n\sqrt{\kappa}+O(te^{-\frac{2}{n}t}).
\end{align}
Thus, by the weighted isoperimetric inequality \eqref{isop-2}, we obtain
\begin{align}\label{mon2-eq3}
   \mathcal{W}(\Sigma_t)=& \int_{\Sigma_t}{H}\,d\mu_t+\int_{\Omega_t}{\overline{Ric}(\partial_r,\partial_r)}\,dv\nonumber\\
   =&n\sqrt{\kappa}\vert\Sigma_t\vert-n\int_{\Omega_t}\frac{\varphi''}{\varphi}dv+O(te^{-\frac{2}{n}t})\vert\Sigma_t\vert\nonumber\\
    \geq&n\sqrt{\kappa}\vert\Sigma_t\vert-n\xi_{1}(\vert\Sigma_t\vert)+O(te^{-\frac{2}{n}t})\vert\Sigma_t\vert\nonumber\\
    =&n\sqrt{\kappa}\vert\widehat{\Sigma}_t\vert-n\xi_{1}(\vert\widehat{\Sigma}_t\vert)+O(te^{-\frac{2}{n}t})\vert\widehat{\Sigma}_t\vert\nonumber\\
    =&\mathcal{W}(\widehat{\Sigma}_t)+O(te^{-\frac{2}{n}t})\vert\widehat{\Sigma}_t\vert.
\end{align}
Now, by the definition of the function $\xi$ and $\vert\widehat{\Sigma}_t\vert=\vert\Sigma_t\vert$, we have
\begin{align}\label{mon2-eq4}
   \mathcal{G}(t)=& \vert\Sigma_t\vert^{-\frac{n-1}{n}}\left(\mathcal{W}(\Sigma_t)-\xi(\vert\Sigma_t\vert)\right)\nonumber\\
   \geq&\vert\widehat{\Sigma}_t\vert^{-\frac{n-1}{n}}\left(\mathcal{W}(\widehat{\Sigma}_t)-\xi(\vert\widehat{\Sigma}_t\vert)+O(te^{-\frac{2}{n}t})\vert\widehat{\Sigma}_t\vert\right)\nonumber\\
   =&O(te^{-\frac{2}{n}t})\vert\widehat{\Sigma}_t\vert^{\frac{1}{n}}
   =O(te^{-\frac{1}{n}t})\vert\Sigma\vert^{\frac{1}{n}},
\end{align}
where in the second inequality we used \eqref{mon2-eq3} and in the last equality we used \eqref{mon2-eq1}, this completes the proof of case $\kappa>0$.

Next, we consider the case $\kappa=0$. At this time, the asymptotic behavior of the IMCF is different from the case $\kappa>0$.
By the definition and the weighted isoperimetric inequality \eqref{isop-2}, we have
\begin{align}
 \mathcal{W}(\Sigma_t)-\xi(\vert\Sigma_t\vert)
 =&\mathcal{W}(\Sigma_t)-\xi(\vert\widehat{\Sigma}_t\vert)\nonumber\\
 =&\int_{\Sigma_t}{H}\,d\mu_t-n\int_{\Omega_t}{\frac{\varphi''}{\varphi}}\,dv-\int_{\widehat{\Sigma}_t}{H}\,d\mu_t+n\int_{\widehat{\Omega}_t}{\frac{\varphi''}{\varphi}}\,dv\nonumber\\
 =&\int_{\Sigma_t}{H}\,d\mu_t-\int_{\widehat{\Sigma}_t}{H}\,d\mu_t-n(\int_{\Omega_t}{\frac{\varphi''}{\varphi}}\,dv-\xi_{1}(\vert\widehat{\Sigma}_t\vert))\nonumber\\
 \geq&\int_{\Sigma_t}{H}\,d\mu_t-\int_{\widehat{\Sigma}_t}{H}\,d\mu_t.\label{eq-Wxigeq}
\end{align}
Therefore, we see that
\begin{align}\label{mon2-eq5}
\mathcal{G}(t)=& \vert\Sigma_t\vert^{-\frac{n-1}{n}}\left(\mathcal{W}(\Sigma_t)-\xi(\vert\Sigma_t\vert)\right)\nonumber\\
\geq&\vert\Sigma_t\vert^{-\frac{n-1}{n}}\int_{\Sigma_t}{H}\,d\mu_t-\vert\widehat{\Sigma}_t\vert^{-\frac{n-1}{n}}\int_{\widehat{\Sigma}_t}{H}\,d\mu_t.
\end{align}
Now, we investigate the limit of the quantity $\vert\Sigma_t\vert^{-\frac{n-1}{n}}\int_{\Sigma_t}{H}\,d\mu_t$. By the estimates \eqref{es-phi H} and \eqref{es-gij}, we have
\begin{align}
|\Sigma_t|^{\frac{n-1}{n}}&=\left(\int_{\mathbb{S}^n}\varphi^n\,dvol_{\mathbb{S}^n}\right)^{\frac{n-1}{n}}\left(1+O(e^{-2\alpha t})\right),\label{es-Sigma}\\
\int_{\Sigma_t}{H}\,d\mu_t&=n\int_{\mathbb{S}^n}{\varphi'\varphi^{n-1}}\,dvol_{\mathbb{S}^n}\left(1+O(e^{-\alpha t})\right).\label{es-inteH}
\end{align}
 Then we calculate as follows:
\begin{align}
\liminf_{t\to\infty}\vert\Sigma_t\vert^{-\frac{n-1}{n}}\int_{\Sigma_t}{H}\,d\mu_t&\geq 
n\liminf_{t\to\infty}\frac{\int_{\mathbb{S}^n}\varphi'\varphi^{n-1}dvol_{\mathbb{S}^n}}{(\int_{\mathbb{S}^n}\varphi^{n}dvol_{\mathbb{S}^n})^{\frac{n-1}{n}}}\notag\\
&\geq n\omega_n^{\frac{1}{n}}\bar{\varphi}'\liminf_{t\to\infty}e^{-2(n-1)\varepsilon(t)}\notag\\
&=n\omega_n^{\frac{1}{n}}\bar{\varphi}',\label{mon2-eq7}
\end{align}
where we have used \eqref{es-phitwo} in the second inequality. On the other hand, since $\widehat{\Sigma}_t$ are radial coordinate spheres, on $\widehat{\Sigma}_t$ we have
\begin{align*}
\vert\widehat{\Sigma}_t\vert=\omega_n\varphi^n(\hat{r}(t)),\quad \int_{\widehat{\Sigma}_t}{H}\,d\mu_t=n\omega_n\varphi'(\hat{r}(t))\varphi^{n-1}(\hat{r}(t)),
\end{align*}
where $\hat{r}(t)$ is the radius of $\widehat{\Sigma}_t$, i.e. $\widehat{\Sigma}_t=\{\hat{r}(t)\}\times\mathbb{S}^n$. Therefore
\begin{equation}\label{eq-Cs}
    \lim_{t\to\infty}\vert\widehat{\Sigma}_t\vert^{-\frac{n-1}{n}}\int_{\widehat{\Sigma}_t}{H}\,d\mu_t=n\omega^{\frac{1}{n}} \lim_{t\to\infty}\varphi'(\hat{r}(t))=n\omega^{\frac{1}{n}}\bar{\varphi}'.
\end{equation}
Combining \eqref{mon2-eq7} and \eqref{eq-Cs} with \eqref{mon2-eq5}, we conclude that $\lim_{t\to\infty}\mathcal{G}(t)\geq0$ in this case.
\end{proof}

Applying Proposition \ref{monprop} and Proposition \ref{limprop}, we can complete the proof of Theorem \ref{Min-In1} now. Since $\mathcal{G}(t)$ is monotone nonincreasing in time $t$, we have
\begin{align*}
    \mathcal{G}(0)\geq\lim_{t\to\infty}\mathcal{G}(t)\geq0,
\end{align*}
Thus, for the strictly mean convex and star-shaped hypersurface $\Sigma$, we obtain
\begin{align}\label{mon2-eq9}
\int_{\Sigma}Hd\mu+\int_{\Omega}{\overline{Ric}(\partial_r,\partial_r)}\,dv\geq\xi(\vert\Sigma\vert).
\end{align}
If the equality holds in the above inequality, then $\mathcal{G}(t)\equiv0$. Since $\mathcal{G}'(0)=0$, it follows from Proposition \ref{monprop} that either: (a) $\Sigma$ is a radial coordinate sphere in $M$, or (b) there exists a radial ball $B(R)$(may be empty) such that $M\setminus B(R)$ is isometric to a space form of constant nonpositive curvature and  $\Sigma$ is a geodesic sphere in it.

Next, we assume that $\Sigma$ is a weakly mean convex and star-shaped hypersurface. By the argument of Li and Wei \cite[Lemma 3.11]{LW17}, we can approximate $\Sigma$ by a family of strictly mean convex and 
star-shaped hypersurfaces using the mean curvature flow. Then the inequality \eqref{mon2-eq9} follows from the approximation. Now we adapt the idea of \cite{G-L2009} to treat the equality case. Assume that $\Sigma$ is weakly mean convex with equality in \eqref{mon2-eq9} attained. Let 
$$\Sigma_+=\{p\in \Sigma|H(p)>0\},$$ 
since there exists at least one elliptic point on $\Sigma$ (see \cite[Lemma 2.1]{LWX-14}), $\Sigma_+$ is open and nonempty. We claim that $\Sigma_+$ is closed. This would implies $\Sigma=\Sigma_+$, so $\Sigma$ is strictly mean convex and we conclude that it is either a radial coordinate sphere, or is isometric to a geodesic sphere in the space form of nonpositive curvature. 

We now prove that $\Sigma_+$ is closed. Pick any $\varpi\in C_0^2(\Sigma_+)$ 
compactly supported in $\Sigma_+$. Let $\Sigma_s$ be the hypersurface determined by $X_s=\exp_{X}{(s\varpi\nu)}$, where $\nu$ is the unit outward normal of $\Sigma$ at $X$. Let $\Omega_s$ be the domain enclosed by $\Sigma_s$ and $\partial M$. It is easy to show $\Sigma_s$ is weakly mean convex and star-shaped when $s$ is small enough. 
Define 
\[\mathcal{L}(\Sigma_s)=\int_{\Sigma_s}Hd\mu+\int_{\Omega_s}{\overline{Ric}(\partial_r,\partial_r)}\,dv-\xi(\vert\Sigma_s\vert).\]
Therefore $\mathcal{L}(\Sigma_s)\geq\mathcal{L}(\Sigma)=0$ for $s$ small, which implies
\begin{align*}
    \frac{d}{ds}\mathcal{L}(\Sigma_s)|_{s=0}=0.
\end{align*}
Simple calculation using Proposition \ref{Prop-general-evo} yields 
\begin{align*}
  \frac{d}{ds}\mathcal{L}(\Sigma_s)|_{s=0}=&\int_{\Sigma}\left(H^2-|A|^2+\overline{Ric}(\partial_r,\partial_r)-\overline{Ric}(\nu,\nu)-\xi'(\vert\Sigma\vert)H\right)\varpi d\mu=0  
\end{align*}
for all $\varpi\in C_0^2(\Sigma_+)$. Thus, 
\begin{align*}
    H^2(p)-|A|^2(p)+\overline{Ric}(\partial_r,\partial_r)(p)-\overline{Ric}(\nu,\nu)(p)-\xi'(\vert\Sigma\vert)H(p)=0, \ \forall p\in\Sigma_+.
\end{align*}
Combining $\overline{Ric}(\partial_r,\partial_r)-\overline{Ric}(\nu,\nu)\leq0$ with Cauchy-Schwarz inequality $|A|^2\geq H^2/n$, we obtain
\begin{align*}
    H(p)\geq\frac{n}{n-1}\xi'(\vert\Sigma\vert)>0, \forall p\in\Sigma_+,
\end{align*}
where the second inequality follows from \eqref{mon-eq9} and hence $\Sigma_+$ is closed. This completes the proof of Theorem \ref{Min-In1}, and then Corollary \ref{Min-InD} follows immediately.
\subsection{Further discussions on Riemannian warped products} In this subsection, we briefly discuss the generalization of our Theorem \ref{Min-In1} to general Riemannian warped products, which need not be warped cylinders. 


Assume that $(V,g_V)$ is a $n$-dimensional closed Riemannian manifold. Fixing $r_0\geq 0$ and let $\lambda=\lambda(r)$ be a smooth positive function defined on the interval $[r_0,\infty)$. Let $(N^{n+1},\bar{g})=[r_0,\infty)\times V^n$ endowed with the following warped product structure
\begin{equation}
    \bar{g}=dr^2+\lambda^2 g_V,\quad r\in[r_0,\infty),
\end{equation}
which satisfies
\begin{align}
    \text{Ric}_V\geq& (n-1)Kg_V,\label{con-Ric}\\
    0\leq (\lambda')^2&-\lambda''\lambda\leq K,\label{con-lambda}
\end{align}
where $K>0$ is a positive constant. Guan, Li and Wang \cite{GLW-2019} proved that if $\Sigma\subset N^{n+1}$ is star-shaped, then the flow hypersurfaces $\Sigma_t$ of the mean curvature type flow
\begin{equation}\label{flow-mean}
    \partial_t X=(n\lambda'-uH)\nu
\end{equation}
starting from $\Sigma$ remains to be star-shaped, exists for all time $t\in[0,\infty)$ and converges exponentially to a radial coordinate sphere as $t\to\infty$. 

Moreover, they showed that $|\Sigma_t|$ is nonincreasing along the flow \eqref{flow-mean}.  By a direct calculation, if $\partial_r\left(\frac{\lambda''}{\lambda}\right)\leq 0$, then the weighted volume $\int_{\Omega_t}{\frac{\lambda''}{\lambda}}\,dv$ is nondecreasing along the flow \eqref{flow-mean}. Thus we can also deduce the inequality \eqref{isop-2}. Then by adding adequate asymptotic conditions to their assumptions \eqref{con-Ric} and \eqref{con-lambda}, we can use \cite[Theorem 1.3]{Scheuer19} to establish the Minkowski inequality \eqref{Mintype} in general Riemannian warped products. The proof is essentially the same as the one we have shown in \S\ref{Sec-4} above. Hence we present an version here without proof and expect it can be improved by weaken the conditions. 
\begin{assu}\label{assG}
    Let $(V,g_V)$ be an $n$-dimensional closed Riemannian manifold and $\lambda=\lambda(r)$ be a smooth positive function defined on the interval $[r_0,\infty)$. Assume that $(N^{n+1},\bar{g})=[r_0,\infty)\times_{\lambda} V^n$ such that $\bar{g}=dr^2+\lambda^2 g_V$. We suppose that the following conditions hold:
    \begin{itemize}
        \item[(i)] $V$ has nonnegative sectional curvature and $\lambda'>0, \lambda''\geq 0$,
        \item[(ii)] $\text{Ric}_V\geq (n-1)Kg_V$,
        \item[(iii)] $\lambda\lambda''\leq(\lambda')^2\leq K+c\lambda^2$,
        \item[(iv)]$\partial_r\left(\frac{\lambda''}{\lambda}\right)\leq 0$,
    \end{itemize} 
    where $K>0$ is a positive constant and $c=\lim_{r\to\infty}\frac{\lambda''}{\lambda}\geq 0$.
\end{assu}
\begin{thm}\label{thm-gene}
    Let $N^{n+1}$ be a Riemannian warped product given in Assumption \ref{assG}, then the Minkowski inequality \eqref{Mintype} holds for any weak mean convex and star-shaped hypersurface $\Sigma$ in $N^{n+1}$. 
\end{thm}
\begin{rem}
    Compared Assumption \ref{assG} with Assumption \ref{ass-1}, we see that in the case of warped cylinders, the constraints in Assumption \ref{assG} is stronger than Assumption \ref{ass-1} by raising the lower bound of $(\lambda')^2$ from $0$ to $\lambda\lambda''$. This leads to some bad consequences. For example, let $\Sigma$ be a weakly mean convex and star-shaped hypersurface in the deSitter-Schwarzschild manifold, Theorem \ref{thm-gene} does not work if $\Sigma$ is close to the horizon $\{0\}\times\mathbb{S}^n$, while Theorem \ref{Min-In1} can deal with this case.
\end{rem}

\section{Proof of Theorem \ref{Min-In2} and Theorem \ref{Thm-Ads}}\label{Sec-5}
In this section, we adapt the idea of \cites{H-2024,Wei18} to complete the proof of Theorem \ref{Min-In2} and Theorem \ref{Thm-Ads}. We firstly deal with the Schwarzschild case. Under our definition of metric \eqref{metric-D} with $\kappa=0$, the Schwarzschild manifold can be seen as
\begin{equation}\label{con-model}
    M^{n+1}=\{x\in\mathbb{R}^{n+1}:|x|\geq \left(\frac{m}{4}\right)^{\frac{1}{n-1}}\},\quad \bar{g}_{ij}(x)=\left(1+\frac{m}{4|x|^{n-1}}\right)^{\frac{4}{n-1}}\delta_{ij}(x).
\end{equation}
Here $\delta_{ij}$ is the canonical Euclidean metric and $|x|$ is the norm of $x$ with respect to the metric $\delta_{ij}$. Consider a map $\mathcal{J}:\mathbb{R}^{n+1}\setminus{\{0\}}\to \mathbb{R}^{n+1}\setminus{\{0\}}$ given by
\begin{equation*}
    \mathcal{J}(x)=\left(\frac{m}{4}\right)^{\frac{2}{n-1}}\cdot \frac{x}{|x|^2},\,\, \forall x\in \mathbb{R}^{n+1}\setminus{\{0\}}.
\end{equation*}
It can been shown that the map $\mathcal{J}$ is a reflection map across the Euclidean sphere $\{x\in\mathbb{R}^{n+1}:|x|=\left(\frac{m}{4}\right)^{\frac{1}{n-1}}\}$ and hence gives an isometry of $\mathbb{R}^{n+1}\setminus{\{0\}}$. Then the double Schwarzschild manifold $(\widehat{M}^{n+1},\widehat{g})$ can be obtained by the combining the Schwarzschild manifold with its image under the reflection map $\mathcal{J}$, which is complete and boundaryless.

We agree on some notations under the model \eqref{con-model} of the Schwarzschild manifold for later use in this section. We set $r_0=\left(\frac{m}{4}\right)^{\frac{1}{n-1}}$ and denote the coordinate centered at the origin of $\mathbb{R}^{n+1}$ as $S(r)$, hence the horizon $\partial M=S(r_0)$. $B(r)$ is denoted as the closure of the domain bounded by $\partial M$ and $S(r)$ for any $r\geq r_0$, while $D(r)$ represents the closed ball centered at the origin of $\mathbb{R}^{n+1}$ with radius $r$ for $r\geq 0$. The volume of a domain $\Omega$ under the metric $\bar{g}$ is denoted as $\mathrm{Vol}(\Omega)$, and the area of a hypersurface $\Sigma$ under the the metric $\bar{g}$ is denoted as $|\Sigma|$ if no ambiguity.

Firstly, we have the following integral formula.
\begin{lem}\label{Lem-5.1}
Suppose that $\Omega$ is a smooth bounded domain in $(M^{n+1},\bar{g})$ and $\Sigma=\partial\Omega\setminus{\partial M}$. Let $u:\Omega^{c}\to\mathbb{R}_{+}$ be a smooth proper function with $u|_{\Sigma}=0$. Let $t>0$, $\Omega_t=\{u< t\}$ and $\Phi:(0,t)\to\mathbb{R}_{+}$ be Lipschitz and compactly supported in $(0,t)$. Then $\zeta=\Phi\circ u:\Omega_t\to\mathbb{R}_{+}$ satisfies:
  \begin{equation}\label{Inte-formula}
      -\int_{\Omega_t}{\nabla\zeta\cdot\nu H}\,dv=\int_{\Omega_t}{\zeta(H^2-|A|^2-\overline{Ric}(\nu,\nu))}\,dv.
  \end{equation}
\end{lem}
\begin{proof}
   The proof is similar as \cite[Lemma 4.1]{Wei18}, hence we omit it here.
\end{proof}
\subsection{Schwarzschild case I: $\Sigma$ is homologous to the horizon}\label{sub-5.1}
\begin{lem}\label{Lem-mohh}
  Let $\Omega$ be a smooth bounded domain in the Schwarzschild manifold $(M^{n+1},\bar{g})$. Suppose that the boundary $\partial\Omega=\Sigma\cup\partial M$ and $\Sigma$ is outward minimizing. Let $\Omega_t$ be the weak solution of IMCF in $\Omega^{c}=M\setminus\Omega$ with initial data $\Omega$. Denote $\Sigma_t=\partial\Omega_t\setminus{\partial M}$, then for all $0\leq \bar{t}<t$, we have
  \begin{equation}\label{Weak-Mono1}
      \mathcal{W}(\Sigma_t)-\xi(\vert\Sigma_t\vert)\leq \mathcal{W}(\Sigma_{\bar{t}})-\xi(\vert\Sigma_{\bar{t}}\vert)+\frac{n-1}{n}\int_{\bar{t}}^{t}{\left[\mathcal{W}(\Sigma_s)-\xi(\vert\Sigma_s\vert)\right]}\,ds.
  \end{equation}
\end{lem}
\begin{proof}
    Using Lemma \ref{Lem-5.1} and following the same argument as in \cite[Lemma 4.2]{Wei18}, we see that for a.e. $0<\bar{t}<t$, there holds
    \begin{align}
    \int_{\Sigma_t}{H}\,d\mu_{t}-\int_{\Sigma_{\bar{t}}}{H}\,d\mu_{\bar{t}}&\leq \frac{n-1}{n}\int_{\bar{t}}^t\int_{\Sigma_s}{H}\,d\mu_s\,ds-\int_{\Omega_t\setminus\Omega_{\bar{t}}}{\overline{Ric}(\nu,\nu)}\,dv\notag\\
    &\leq \frac{n-1}{n}\int_{\bar{t}}^t\int_{\Sigma_s}{H}\,d\mu_s\,ds-\int_{\Omega_t\setminus\Omega_{\bar{t}}}{\overline{Ric}(\partial_r,\partial_r)}\,dv,\label{mono-1}
    \end{align} 
where we have used the fact that $\overline{Ric}(\partial_r,\partial_r)\leq \overline{Ric}(\nu,\nu)$ in the second inequality of \eqref{mono-1}.
By the definition \eqref{defn-W} of $\mathcal{W}$, we see that
\begin{equation}\label{mono-3}
\mathcal{W}(\Sigma_t)-\mathcal{W}(\Sigma_{\bar{t}})\leq \frac{n-1}{n}\int_{\bar{t}}^{t}\int_{\Sigma_s}{H}\,d\mu_s\,ds
\end{equation}
holds for a.e. $0<\bar{t}<t$.
Since $\Sigma$ is outward minimizing, then $|\Sigma_t|=e^t|\Sigma|$ for all $t\geq 0$ by property (5) in Theorem \ref{Prop-Weak}, which follows that $\frac{d}{dt}|\Sigma_t|=|\Sigma_t|$. Hence 
\begin{align}
    \xi(\vert\Sigma_t\vert)-\xi(\vert\Sigma_{\bar{t}}\vert)&=\int_{\bar{t}}^t{\frac{d}{ds}\xi(|\Sigma_s|)}\,ds\notag\\
    &=\int_{\bar{t}}^t{\xi'(|\Sigma_s|)|\Sigma_s|}\,ds\notag\\
    &=\frac{n-1}{n}\int_{\bar{t}}^t{\xi(|\Sigma_s|)}\,ds+(n-1)\int_{\bar{t}}^t{\xi_1(|\Sigma_s|)}\,ds,\label{mono-2}
\end{align}
where we have used \eqref{mon-eq9} in the last equality. Combining \eqref{mono-3} with \eqref{mono-2} gives 
\begin{align}
&\left[\mathcal{W}(\Sigma_t)-\xi(\vert\Sigma_t\vert)\right]-\left[ \mathcal{W}(\Sigma_{\bar{t}})-\xi(\vert\Sigma_{\bar{t}}\vert)\right]\notag\\
\leq&\frac{n-1}{n}\int_{\bar{t}}^t{\left[\mathcal{W}(\Sigma_s)-\xi(\vert\Sigma_s\vert)\right]}\,ds+(n-1)\int_{\bar{t}}^t{\left[\int_{\Omega_s}{\frac{\varphi''}{\varphi}}\,ds-\xi_1(|\Sigma_s|)\right]}\,ds\notag\\
\leq&\frac{n-1}{n}\int_{\bar{t}}^t{\left[\mathcal{W}(\Sigma_s)-\xi(\vert\Sigma_s\vert)\right]}\,ds,\quad\mathrm{a.e.} \,\, 0<\bar{t}<t.\label{mono-xi}
\end{align}

To show \eqref{mono-xi} holds for all pairs $0\leq\bar{t}<t$, we use the $C^{1,\beta}$ convergence \eqref{C-1beta-con} and the weak convergence of the mean curvature. For any $\bar{t}\geq 0$, we can find a sequence of time $t_i\searrow \bar{t}$, such that $0\leq\bar{t}<t_i<t$ satisfies \eqref{mono-xi}. By \eqref{C-1beta-con}, we have $\Sigma_{t_i}\to\Sigma_t'$ in $C^{1,\beta}$ away from a singular set $Z$ of Hausdorff dimension at most $n-7$. This implies that
\begin{align}\label{eq-lim-area} \lim_{i\to\infty}\vert\Sigma_{t_i}\vert=\vert\Sigma'_{\bar{t}}\vert,\ \lim_{i\to\infty}\int_{\Omega_{t_i}}\overline{Ric}(\partial_r,\partial_r)\,dv=\int_{\Omega'_{\bar{t}}}\overline{Ric}(\partial_r,\partial_r)\,dv.
\end{align} 
On the other hand, since the weak mean curvature of $\Sigma_{t_i}$ equals to $|\nabla u|$ a.e. and is uniformly bounded for a.e. $x\in\Sigma_{t_i}$. Then it follows from the Riesz Representation theorem (\cite[(1.13)]{HI-2001}) that 
\begin{equation}\label{eq-Riesz}
 \int_{\Sigma_{t_i}}{H_{\Sigma_{t_i}}\nu_{\Sigma_{t_i}}\cdot X}\,d\mu_{t_i}\to \int_{\Sigma_{\bar{t}}'}{H_{\Sigma_{\bar{t}}'}\nu_{\Sigma_{\bar{t}}'}\cdot X}\,d\mu_{\bar{t}}',\quad X\in C^0_c(TM),
\end{equation}
which means 
\begin{equation}\label{eq-H-con1}
   \lim_{i\to\infty}\int_{\Sigma_{t_i}}{H}\,d\mu_{t_i}= \int_{\Sigma_{\bar{t}}'}{H}\,d\mu_{\bar{t}}'.
\end{equation}
Now, combining \eqref{eq-lim-area} and \eqref{eq-H-con1} with \eqref{mono-xi}, we conclude that
\begin{align}\label{eq-lim-xi}
    \left[\mathcal{W}(\Sigma_t)-\xi(\vert\Sigma_t\vert)\right]-\left[ \mathcal{W}(\Sigma'_{\bar{t}})-\xi(\vert\Sigma'_{\bar{t}}\vert)\right]
\leq\frac{n-1}{n}\int_{\bar{t}}^t{\left[\mathcal{W}(\Sigma_s)-\xi(\vert\Sigma_s\vert)\right]}\,ds
\end{align}
holds for all $\bar{t}\geq 0$ and a.e. $t>0$ with $\bar{t}<t$.

Next, by \cite[(1.15)]{HI-2001}, the weak mean curvature of $\Sigma_{\bar{t}}$ and $\Sigma_{\bar{t}}'$ satisfy
\begin{align*}
    H_{\Sigma_{\bar{t}}'}&=0 \quad \,\text{on} \quad \Sigma_{\bar{t}}'\setminus \Sigma_{\bar{t}},\\
    H_{\Sigma_{\bar{t}}'}&=H_{\Sigma_{\bar{t}}}\geq 0 \quad \mathcal{H}^n\text{-a.e.} \quad \text{on} \quad \Sigma_{\bar{t}}'\cap\Sigma_{\bar{t}}.
\end{align*}
Since the weak mean curvature $H$ is nonnegative on $\Sigma_{\bar{t}}$, then we have
\begin{equation}\label{compare-s-s'}
  \int_{\Sigma_{\bar{t}}'}{H}\,d\mu_{\bar{t}}'\leq \int_{\Sigma_{\bar{t}}}{H}\,d\mu_{\bar{t}}.
\end{equation}
Moreover, by the property (1) in Theorem \ref{Prop-Weak} and the fact that $\Omega_{\bar{t}}\subset\Omega_{\bar{t}}'$, $\overline{Ric}(\partial_r,\partial_r)<0$, we deduce that 
\begin{equation}\label{compare-5.2-2}
    |\Sigma_{\bar{t}}|\leq |\Sigma_{\bar{t}}'|,\quad \mathcal{W}(\Sigma_{\bar{t}}')\leq \mathcal{W}(\Sigma_{\bar{t}}).
\end{equation}
Thus by \eqref{eq-lim-xi}-\eqref{compare-5.2-2}, we conclude that \eqref{mono-xi} holds for all $\bar{t}\geq 0$ and a.e. $t>0$ with $\bar{t}<t$. 

Similarly, for any $t>0$ with $0\leq\bar{t}<t$, we can find a sequence of time $t_i\nearrow t$ such that $0\leq\bar{t}<t_i<t$ satisfies \eqref{mono-xi}. In this case, we have $\Sigma_{t_i}\to\Sigma_t$ in $C^{1,\beta}$ away from a singular set $Z$ of Hausdorff dimension at most $n-7$ by \eqref{C-1beta-con}, and also the Riesz Representation theorem gives
\begin{equation}
    \lim_{i\to\infty}\int_{\Sigma_{t_i}}{H}\,d\mu_{t_i}=\int_{\Sigma_{\bar{t}}}{H}\,d\mu_{\bar{t}},
\end{equation}
which yields \eqref{mono-xi} holds for any $t>0$ with $0\leq\bar{t}<t$. This completes the proof of Lemma \ref{Lem-mohh}.
\end{proof}

\begin{prop}\label{prop-g}
    Under the assumptions in Lemma \ref{Lem-mohh}, the quantity $\mathcal{G}(t)$ is nonincreasing for all $t\geq 0$. Moreover, if $\mathcal{G}(t)=\mathcal{G}(\bar{t})$ for some pair $0\leq \bar{t}<t$, then $\Sigma_s$ is a radial coordinate sphere (possibly away from a set of measure zero if $n\geq 7$) for every $s\in(\bar{t},t]$.  
\end{prop}
\begin{proof}
   By \eqref{Weak-Mono1} and Gronwall's lemma, we have
   \begin{equation}
       \mathcal{W}(\Sigma_t)-\xi(\vert\Sigma_t\vert)\leq \left[\mathcal{W}(\Sigma_{\bar{t}})-\xi(\vert\Sigma_{\bar{t}}\vert)\right]e^{\frac{n-1}{n}(t-\bar{t})}
   \end{equation}
for all $0\leq\bar{t}<t$. Since $|\Sigma_t|=e^t|\Sigma|$,  $\mathcal{G}(t)$ is nonincreasing for all $t$. If $Q(t)=Q(\bar{t})$ for some pair $0<\bar{t}<t$, then from inequality \eqref{mono-xi}, we see that 
\begin{equation*}
\int_{\Omega_s}\frac{\varphi''}{\varphi}\,ds=\xi_1(|\Sigma_s|)
\end{equation*}
holds for a.e. $s\in[\bar{t},t]$, which implies that $\Sigma_s$ is a radial coordinate sphere (possibly away from a set of measure zero if $n\geq 7$) for a.e. $s\in[\bar{t},t]$ by the argument in Theorem \ref{thm-isop} and Remark \ref{rem-reduce}. Due to property (3) in Theorem \ref{Prop-Weak}, $\Sigma_s$ is a radial coordinate sphere (possibly away from a set of measure zero if $n\geq 7$) for every $s\in(\bar{t},t]$.
\end{proof}
\begin{prop}\label{prop-lim}
    We have $\lim_{t\to\infty}{\mathcal{G}(t)}\geq 0$.
\end{prop}
\begin{proof}
    Let $\widehat{\Sigma}$ be the radial coordinate sphere in the Schwarzschild manifold $(M^{n+1},\bar{g})$ which satisfies $|\widehat{\Sigma}|=|\Sigma|$, and $\widehat{\Omega}$ be the domain enclosed by $\widehat{\Sigma}$ and $\partial M$. Let $\widehat{\Omega}_t$ be the classical solutions to the IMCF, then we have
    \begin{equation*} |\widehat{\Sigma}_t|=e^t|\widehat{\Sigma}|=e^t|\Sigma|=|\Sigma_t|,
    \end{equation*}
    where the first equality follows from the evolution equation \eqref{eveq-area} of the area element along the smooth IMCF and the last equality follows from property (5) in Theorem \ref{Prop-Weak}. Then the same calculation as in \eqref{eq-Wxigeq} gives
    \begin{equation}\label{eq-5.4-1}
        \mathcal{W}(\Sigma_t)-\xi(\vert\Sigma_t\vert)\geq \int_{\Sigma_t}{H}\,d\mu_t-\int_{\widehat{\Sigma}_t}{H}\,d\mu_t.
    \end{equation}
    Using the model \eqref{con-model} of the Schwarzschild manifold and we define the blow down object by 
    \begin{equation*}
        \Sigma_t^{\lambda}:=\lambda\Sigma_t=\{\lambda x:x\in\Sigma_t\},\quad g^{\lambda}(x):=\lambda^2g(x/{\lambda}),
    \end{equation*}
    where $\lambda>0$ is a fixed number. Define $r(t)$ by $|\Sigma_t|=\omega_n r(t)^n$, then $|\Sigma_t^{1/{r(t)}}|_{g^{1/{r(t)}}}=\omega_n$ and \cite[Lemma 7.1]{HI-2001} implies that 
    \begin{equation}\label{eq-5.4-2}
        \Sigma_t^{1/{r(t)}}\to S(1)\quad \text{in}\,\,\, C^{1,\alpha}\,\,\, \text{as}\,\,\,t\to\infty.
    \end{equation}
    According to \cite[Equation (7.2)]{HI-2001}, there exists positive constants $C$ and $R_0$ depending only on $n$, such that 
    \begin{equation*}
        |\nabla u(x)|\leq \frac{C}{|x|},\quad \text{for all}\,\, |x|\geq R_0.
    \end{equation*}
Then by property (4) in Theorem \ref{Prop-Weak}, we have
\begin{equation}\label{eq-5.4-3}
    H=|\nabla u|\leq\frac{C}{|x|}\leq \frac{C'}{r(t)},\quad \text{a.e. on} \,\,\Sigma_t
\end{equation}
for a.e. sufficiently large $t$ and $C'=C+1$, where we have used the convergence result \eqref{eq-5.4-2}. Then the mean curvature of $\Sigma_t^{1/{r(t)}}$ with respect to the metric $g^{1/{r(t)}}$ satisfies
\begin{equation}\label{eq-5.4-5}
    H^{1/{r(t)}}=r(t)H(r(t)x)\leq C',\quad \text{a.e. on} \,\,\Sigma_t^{1/{r(t)}}
\end{equation}
for a.e. sufficiently large $t$. Then for any sequence $t_i\to\infty$ such that \eqref{eq-5.4-5} holds, we write $\Sigma_t^{1/{r(t_i)}}$ as graphs of $C^{1,\alpha}$ functions over $S(1)$ and have the weak convergence of the mean curvature 
\begin{equation}\label{eq-5.4-6}
    \int_{\Sigma_{t_i}^{1/{r(t_i)}}}{H_{\Sigma_{t_i}^{1/{r(t_i)}}}\nu_{\Sigma_{t_i}^{1/{r(t_i)}}}\cdot X}\to\int_{S(1)}{H_{S(1)}\nu_{S(1)}\cdot X},\quad X\in C_c^0(TM).
\end{equation}
Then by \eqref{eq-5.4-2}, \eqref{eq-5.4-5} and \eqref{eq-5.4-6}, we have
\begin{align}
    \lim_{t_i\to\infty}|\Sigma_{t_i}|^{-\frac{n-1}{n}}\int_{\Sigma_{t_i}}{H}\,d\mu_{t_i}&=\omega_n^{-\frac{n-1}{n}}\lim_{t_i\to\infty}{r(t_i)^{-(n-1)}}\int_{\Sigma_{t_i}}{H}\,d\mu_{t_i}\notag\\
    &=\omega_n^{-\frac{n-1}{n}}\lim_{t_i\to\infty}\int_{\Sigma_{t_i}^{1/{r(t_i)}}}{H^{1/{r(t_i)}}}(x)\,d\mu_{{\Sigma_{t_i}^{1/{r(t_i)}}}}\notag\\
    &=n\omega_n^{\frac{1}{n}}.\label{eq-5.4-7}
\end{align}
Combining \eqref{eq-5.4-7} with \eqref{eq-Cs} and \eqref{eq-5.4-1} completes the proof of Proposition \ref{prop-lim}, since we have $\lim_{r\to\infty}\varphi'=1$ in the Schwarzschild case.
\end{proof}
Up to here, we have shown that inequality \eqref{Mintype2} holds if $\Sigma$ is homologous to the horizon. We now check the equality case. If the equality holds in \eqref{Mintype2}, then $\Sigma_t$ is a radial coordinate sphere (possibly away from a set of measure zero if $n\geq 7$) for every $t>0$ by Proposition \ref{prop-g}. Property (3) in Theorem \ref{Prop-Weak} says that $\Sigma_t\to\Sigma'$ locally in $C^{1,\beta}$ as $t\to 0+$ (possibly away from a set of measure zero if $n\geq 7$) and hence $\Sigma'$ is also a radial coordinate sphere (possibly away from a set of measure zero if $n\geq 7$). Then $\Sigma'$ has positive constant mean curvature (a.e. if $n\geq 7$). Applying \cite[(1.15)]{HI-2001} again, we have
\begin{align*}
    H_{\Sigma'}&=0 \quad \,\text{on} \quad \Sigma'\setminus \Sigma,\\
    H_{\Sigma'}&=H_{\Sigma}\geq 0 \quad \mathcal{H}^n\text{-a.e.} \quad \text{on} \quad \Sigma'\cap\Sigma,
\end{align*}
which yields $\Sigma=\Sigma'$ (a.e. if $n\geq 7$), and hence $\Sigma$ is a radial coordinate sphere (possibly away from a set of measure zero if $n\geq 7$). Since $\Sigma$ is smooth, we conclude that $\Sigma$ is precisely a radial coordinate sphere.

\subsection{Schwarzschild case II: $\Sigma$ is null homologous}\label{Subsec-5.2}
In this subsection, we assume that $\Sigma$ is null homologous. We fill-in the region $W$ bounded by the horizon $\partial M$ as in \cite[\S 6]{HI-2001} to get a new space $\tilde{M}$, and run the weak IMCF in $\tilde{M}$ until it nearly touches $\partial M$, then we jump to the strictly minimizing hull $F$ of $\Omega_t\cup W$ and restart the weak IMCF from $F$.

Assume that $t_1$ is the jump time, we firstly deal with the time interval $t\in[0,t_1]$, where $\Omega_t$ is always null homologous. However, the argument as in \S\ref{sub-5.1} cannot directly apply, because the last equality in \eqref{mono-xi} is not valid in general since $\Omega_t$ is null homologous. We firstly recall a result proved by Brendle and Eichmair \cite{BE13}:
\begin{thmE}[\cite{BE13}]
Let $\Omega$ be an isoperimetric region in the double Schwarzschild manifold. If the volume of $\Omega$ is sufficiently large, then $\Omega$ contains the horizon $\partial M$ and is bounded by two radial coordinate spheres.
\end{thmE}

Since the double Schwarzschild manifold can be obtained by combining the Schwarzschild manifold with its image under the reflection map $\mathcal{J}$ across the horizon $S(r_0)$, their proof of Theorem E also suits the case of Schwarzschild manifold by a slightly modification. We need the following useful proposition.

\begin{thmF}[\cite{EM13}]\label{Prop-off}
    Given $(\tau,\eta)\in(1,\infty)\times(0,1)$, there exists $V_0>0$ so that the following holds: Let $\Omega$ be a bounded set with finite perimeter in the Schwarzschild manifold $({M}^{n+1},\bar{g})$ containing the horizon $\partial M$, and let $r\geq r_0$ be such that 
    \begin{equation*}
        \mathrm{Vol}(\Omega)=\mathrm{Vol}(B(r))\geq V_0.
    \end{equation*}
    If $\Omega$ is $(\tau,\eta)$-off-center, i.e., if $|\partial^*\Omega\setminus {B}(\tau r)|\geq \eta|S(r)|$, then
    \begin{equation}\label{In-offcenter}
        |\partial^*\Omega|\geq |S(r)|+\frac{c\eta m}{2}\left(1-\frac{1}{\tau}\right)^2 r.
    \end{equation}
    Here, $c>0$ is a constant that only depends on $n$.
\end{thmF}

\begin{thm}\label{Thm-isolarge}
    There exists $V_1>0$ with the following property: If $\Omega$ is an isoperimetric region in the Schwarzschild manifold with $\mathrm{Vol}(\Omega)\geq V_1$, then $\partial^{*}\Omega$ is the union of the horizon and a radial coordinate sphere.
\end{thm}
\begin{proof}
    Suppose that the conclusion is false. Then there exists a sequence of isoperimetric regions $\Omega_k$ with $\mathrm{Vol}(\Omega_k)\to\infty$ such that $\partial^{*}\Omega_k$ is not the union of the horizon and a radial coordinate sphere. Let $r_k\geq r_0$ be such that $\mathrm{Vol}(\Omega_k)=\mathrm{Vol}(B({r_k}))$. Since $\Omega_k$ is an isoperimetric region, it follows that 
    \begin{equation}\label{In-isoo}
        |\partial^*{\Omega_k}|\leq|S({r_k})|+|\partial M|.
    \end{equation}
    Then we consider two cases.

Case 1: Suppose that for every $\tau>1$, we have
\begin{equation*}
    \liminf_{k\to\infty}{r_k^{-n-1}\text{Vol}(\Omega_k\setminus B(\tau r_k))}=0.
\end{equation*}
As in the proof of Theorem 5.1 in \cite{EM13-2}, we see that after passing to a subsequence if necessary, $\Omega_k\subset B(2r_k)$ and the rescaled regions $r_k^{-1}\Omega_k$ converge to the  $D(1)\setminus{\{0\}}$ in $\mathbb{R}^{n+1}$. Then we have
\begin{equation}
    B(r_k/2)\setminus B(r_k/4)\subset \Omega_k\subset B(2r_k)
\end{equation}
for some large integer $k$. By the regularity theorem (cf. \cite[Proposition 2.4]{Manuel-04}), $\partial^{*}\Omega_k\setminus B(r_k/2)$ is a smooth, embedded hypersurface with constant mean curvature. Brendle proved that any closed, embedded hypersurface with constant mean curvature in the deSitter-Schwarzschild manifold is a radial coordinate sphere (see \cite[Corollary 1.2]{Br13}),
so we denote $\partial^{*}\Omega_k\setminus B(r_k/2)$ as $S(\bar{r}_k)$. Set $\hat{r}_k=\inf\{r\in(0,r_k/2):B(r_k/2)\setminus B(r)\subset \Omega_k\}$. Then by the half space theorem, the radial coordinate sphere $S(\hat{r}_k)$ intersects $\overline{\partial^{*}\Omega_k}$ in the regular set of $\partial^{*}\Omega_k$, and then the maximum principle implies that $S(\hat{r}_k)\subset\partial^{*}\Omega_k$, hence $\Omega_k$ has a connected component which is enclosed by $S(\bar{r}_k)$ and $S(\hat{r}_k)$. Note that if $\hat{r}_k>r_0$, then the region enclosed by $S(\bar{r}_k)$ and $S(\hat{r}_k)$ cannot be isoperimetric, so does $\Omega_k$. Then $\hat{r}_k=r_0$, this means $\Omega_k$ is the domain enclosed by $\partial M$ and $S(\bar{r}_k)$,
which contradicts the choice of $\Omega_k$.  

Case 2: Assume that there exists a real number $\tau>1$ such that 
\begin{equation*}
    \liminf_{k\to\infty}{r_k^{-n-1}\mathrm{Vol}(\Omega_k\setminus B(\tau r_k))}>0.
\end{equation*}
By \cite[Lemma 2.4]{EM13-2}, we have
\begin{equation*}
\liminf_{k\to\infty} r_k^{-n}|\partial^{*}\Omega_k\setminus B(\tau r_k)|>0
\end{equation*}
for some $\tau>1$. Hence we can find a real number $\eta\in(0,1)$, such that the sets $\Omega_k\cup\partial M$ are $(\tau,\eta)$-off-center as in Proposition A when $k$ is sufficiently large, then we have
\begin{equation*}
    |\partial^*\Omega_k|\geq |S(r_k)|+\frac{c\eta m}{2}\left(1-\frac{1}{\tau}\right)^2 r_k-|\partial M|,
\end{equation*}
which contradicts with \eqref{In-isoo} if $k$ is sufficiently large. 
\end{proof}

We now turn back to our proof. According to Theorem \ref{Thm-isolarge}, if $\mathrm{Vol}(\Omega)\geq V_1$ and the domain $B(\tilde{r})$ has the same volume of $\Omega$, then we have
\begin{equation*}
  |\Sigma|\geq |\partial B(\tilde{r})|=|S(\tilde{r})|+|\partial M|=\xi_0^{-1}(\mathrm{Vol}(B(\tilde{r})))+|\partial M|=\xi_0^{-1}(\mathrm{Vol}(\Omega))+|\partial M|,
\end{equation*}
where $\xi_0$ is defined in the isoperimetric inequality \eqref{isop2-eq1}.
This means
\begin{equation}\label{ineq-iso2}
    \mathrm{Vol}(\Omega)\leq\xi_0(|\Sigma|-|\partial M|).
\end{equation}
Combining this with \eqref{isop2-eq2} gives
\begin{align*}
   \int_{\Omega}{\frac{\varphi''}{\varphi}}\,ds \leq \xi_{1}(|\Sigma|-|\partial M|).
\end{align*}
Then the same argument as in the proof of Lemma \ref{Lem-mohh} goes through except that \eqref{mono-xi} now becomes
 \begin{align}
&\left[\mathcal{W}(\Sigma_t)-\xi(\vert\Sigma_t\vert)\right]-\left[ \mathcal{W}(\Sigma_{\bar{t}})-\xi(\vert\Sigma_{\bar{t}}\vert)\right]\notag\\
\leq&\frac{n-1}{n}\int_{\bar{t}}^t{\left[\mathcal{W}(\Sigma_t)-\xi(\vert\Sigma_t\vert)\right]}\,ds+(n-1)\int_{\bar{t}}^t{\left[\int_{\Omega_s}{\frac{\varphi''}{\varphi}}\,ds-\xi_1(|\Sigma_s|)\right]}\,ds\notag\\
\leq&\frac{n-1}{n}\int_{\bar{t}}^t{\left[\mathcal{W}(\Sigma_t)-\xi(\vert\Sigma_t\vert)\right]}\,ds+(n-1)\int_{\bar{t}}^t\left[\xi_1(|\Sigma_s|-|\partial M|)-\xi_1(|\Sigma_s|)\right]ds\\
<&\frac{n-1}{n}\int_{\bar{t}}^t{\left[\mathcal{W}(\Sigma_t)-\xi(\vert\Sigma_t\vert)\right]}\,ds,\label{mono-xi1}
\end{align}
from which we deduce that $\mathcal{G}(t)$ is strictly decreasing for $t\leq t_1$. 

If $n<7$, when $\Omega_{t_1}$ jumps to $F$, the strictly minimizing hull of $\Omega_{t_1}\cup W$, then the same argument as in \cite[\S 4.2]{Wei18} yields
\begin{equation}\label{eq-quan}
\partial F\in C^{1,\alpha},\quad |\partial F|\geq |\Sigma_{t_1}|,\quad \int_{\partial F}{H}\,d\mu\leq \int_{\Sigma_{t_1}}{H}\,d\mu_{t_1}.
\end{equation}
Meanwhile, since $t_1$ is the jump time, we have $\Omega_{t_1}\subset F\setminus W$ and hence
\begin{equation}\label{eq-Ric}
    \int_{\Omega_{t_1}}{\overline{Ric}(\partial_r,\partial_r)}\,dv\geq \int_{F\setminus W}{\overline{Ric}(\partial_r,\partial_r)}\,dv.
\end{equation}
Moreover, $F$ is the suitable condition to restart the flow. Then by the same argument in \S\ref{sub-5.1}, $\mathcal{G}(t)$ starting from initial value $F$ is nonincreasing and satisfies $\lim_{t\to\infty}\mathcal{G}(t)\geq 0$, and in turn
\begin{equation}\label{eq-afterjump}
      \mathcal{W}(\partial F)-\xi(|\partial F|)\geq 0.
\end{equation}
Combining \eqref{eq-quan}-\eqref{eq-afterjump}, we see that the quantity $\mathcal{G}(t)$ is nonincreasing during the jump and hence $\Sigma$ satisfies the strict inequality \eqref{eq-afterjump}. We summarize the results of this subsection in the following theorem.
\begin{thm}\label{In-Min-null}
    Let $\Omega$ be a bounded domain with smooth and outward minimizing boundary $\Sigma$ in the Schwarzschild manifold $(M^{n+1},\bar{g})$. If $n<7$ and the volume of $\Omega$ is sufficiently large, then we have
    \begin{equation*}
        \int_{\Sigma}{H}\,d\mu+\int_{\Omega}{\overline{Ric}(\partial_r,\partial_r)}\,dv>\xi(|\Sigma|),
    \end{equation*}
    where $\xi$ is the associated monotonically increasing function.
\end{thm}
Combining the results in \S\ref{sub-5.1} and Theorem \ref{In-Min-null} completes the proof of Theorem \ref{Min-In2}.
\subsection{Hyperbolic case} In this subsection, we complete the proof of Theorem \ref{Thm-Ads}. Note that we also have Lemma \ref{Lem-5.1}, Lemma \ref{Lem-mohh} and Proposition \ref{prop-g} in this case. Due to the lack of blow down lemma, we cannot apply the method in Proposition \ref{prop-lim} to obtain the limit of $\mathcal{G}(t)$ as $t$ tends to $\infty$. However, \cite[Theorem 1.1]{H-2024} says that the weak solution $\Sigma_t$ to the IMCF becomes star-shaped and smooth after a large time and hence we can directly apply the result of Proposition \ref{limprop} to deduce the limit. Then we see that inequality \eqref{Mintype4} holds for $\Sigma$ which is outward minimizing in $\mathbb{H}^{n+1}$ with $2\leq n<7$, and the equality is characterized by the isoperimetric inequality, which means that $\Sigma$ is a geodesic sphere. This completes the proof of Theorem \ref{Thm-Ads}.


\end{sloppypar}

\begin{thebibliography}{9}
	\bib{Alex-1937}{article}{
		author={A. D. Aleksandrov},
		title={Zur Theorie der gemischten Volumina von konvexen Körpern, II: Neue Ungleichungen zwischen den gemischten Volumina und ihre Anwendungen (in Russian)},
		journal={Mat. Sbornik N. S.},
		volume={2},
		year={1937},
		pages={1205--1238},
	}

   	\bib{Alex-1938}{article}{
   	author={A. D. Aleksandrov},
   	title={Zur Theorie der gemischten Volumina von konvexen Körpern, IV: Die gemischten Diskriminanten und die gemischten Volumina (in Russian)},
   	journal={Mat. Sbornik N. S.},
   	volume={3},
   	year={1938},
   	pages={227--251},
   }


	
	\bib{Bray97}{article}{
		author={Bray, Hubert},
		title={The Penrose inequality in general relativity and volume comparison
			theorems involving scalar curvature},
		journal={ PhD thesis, Stanford University},
		year={1997},
		pages={arXiv:0902.3241 },
	}
	\bib{BM2002}{article}{
		author={Bray, Hubert},
		author={Morgan, Frank},
		title={An isoperimetric comparison theorem for Schwarzschild space and other manifolds},
		journal={Proc. Amer. Math. Soc},
		volume={130},
		year={2002},
		number={5},
		pages={1467--1472},
	}
 \bib{Br13}{article}{
    author = {Brendle, Simon},
     title = {Constant mean curvature surfaces in warped product manifolds},
   journal = {Publ. Math. Inst. Hautes Études Sci.},
    volume = {117},
      year = {2013},
     pages = {247--269},
}
\bib{BE13}{article}{
    author = {Brendle, Simon},
    author = {Eichmair, Michael},
     title = {Isoperimetric and {W}eingarten surfaces in the Schwarzschild manifold},
   journal = {J. Differential Geom.},
    volume = {94},
      year = {2013},
    number = {3},
     pages = {387--407},
}
\bib{BHW16}{article}{
	author = {Brendle, Simon},
	author = {Hung, Pei-Ken},
	author = {Wang, Mu-Tao},
	title = {A Minkowski inequality for hypersurfaces in the anti--de Sitter--Schwarzschild manifold},
	journal = {Comm. Pure Appl. Math.},
	volume = {69},
	year = {2016},
	number = {1},
	pages = {124--144},
}
\bib{BGL}{article}{
	author = {Brendle, Simon},
	author = {Guan, Pengfei},
	author = {Li, Junfang},
	title = {An inverse curvature type hypersurface flow in space forms},
	year = {preprint},
}

\bib{CLZ19}{article}{
	author={Chen, Daguang},
	author={Li, Haizhong},
	author={Zhou, Tailong},
	title={A Penrose type inequality for graphs over Reissner-Nordström–anti-deSitter manifold},
	journal={J. Math. Phys},
	volume={60},
	date={2019},
	number={4},
	pages={043503, 12pp},
}

\bib{Cho15}{article}{
	author={Chodosh, Otis},
	title={The geometry of asymptotically
 hyperbolic manifolds},
	journal={Ph. D. dissertation(Stanford University, 2015)},
}

\bib{Cho16}{article}{
  author={Chodosh, Otis},
	title={Large isoperimetric regions in asymptotically hyperbolic manifolds},
	journal={Comm. Math. Phys.},
	volume={343},
	year={2016},
	number={2},
	pages={393--443},
}

\bib{CGGK0722}{article}{
  author={Corvino, Justin},
	author={Gerek, Aydin},
	author={Greenberg, Michael},
	author={Krummel, Brian},
	title={On isoperimetric surfaces in general relativity},
	journal={Pacific J. Math.},
	volume={231},
	year={2007},
	number={1},
	pages={63--84},
}

\bib{EM13}{article}{
	author={Eichmair, Michael},
	author={Metzger, Jan},
	title={Unique isoperimetric foliations of asymptotically flat manifolds in all dimensions},
	journal={Invent. Math.},
	volume={194},
	year={2013},
	number={3},
	pages={591--630},
}

\bib{EM13-2}{article}{
	author={Eichmair, Michael},
	author={Metzger, Jan},
	title={Large isoperimetric surfaces in initial data sets},
	journal={J. Differential Geom.},
	volume={94},
	year={2013},
	number={1},
	pages={159--186},
}

 \bib{Fenchel-1936}{article}{
 	author={W. Fenchel},
 	title={Inégalités quadratiques entre les volumes mixtes des corps convexes},
 	journal={C. R. Acad. Sci. Paris},
 	volume={203},
 	year={1936},
 	pages={647--650},
 }

 \bib{GWW14}{article}{
 	author={Ge, Yuxin},
 	author={Wang, Guofang},
 	author={Wu, Jie},
 	title={Hyperbolic Alexandrov-Fenchel quermassintegral inequalities II},
 	journal={ J. Differential Geom. },
 	volume={98},
 	year={2014},
 	number={2},
 	pages={237--260},
 }

\bib{GWWX-15}{article}{
	author={Ge, Yuxin},
        author={Wang, Guofang},
        author={Wu, Jie},
        author={Xia, Chao},
	title={A Penrose inequality for graphs over Kottler space},
	journal={Calc. Var. Partial Differential Equations},
	volume={52},
	year={2015},
	number={3-4},
	pages={755--782},
} 
 
\bib{Ger1990}{article}{
	author={Gerhardt, Claus},
	title={Flow of nonconvex hypersurfaces into spheres},
	journal={ J. Differential Geom. },
	volume={32},
	year={1990},
	number={1},
	pages={ 299--314},
}


\bib{G-L2009}{article}{
	author={Guan, Pengfei},
	author={Li, Junfang},
	title={The quermassintegral inequalities for $k$-convex starshaped domains},
	journal={Adv. Math.},
	volume = {221},
	pages = {1725--1732},
	year = {2009},
}
\bib{GL15}{article}{
	author={Guan, Pengfei},
	author={Li, Junfang},
	title={A mean curvature type flow in space forms},
	journal={Int. Math. Res. Not.},
	year={2015},
	volume={13},
	pages={4716--4740},
}

 
 \bib{GLW-2019}{article}{
 author={Guan, Pengfei},
 author={Li, Junfang},
 author={Wang, Mu-Tao},
 title={A volume preserving flow and the isoperimetric problem in warped product spaces},
 journal={Trans. Amer. Math. Soc.},
 year={2019},
 volume={372},
 number={4},
 pages={2777--2798},
 }

\bib{H-2024}{article}{
 author={Harvie, Brian},
 title={On weak inverse mean curvature flow and Minkowski-type inequalities in hyperbolic space},
 journal={arXiv: 2404.08410v3},
 }

 \bib{HW-2024}{article}{
 author={Harvie, Brian},
 author={Wang, Ye-Kai},
 title={Quasi-Spherical metrics and the static Minkowski inequality},
 journal={arXiv: 2403.06216},
 }




\bib{HI-2001}{article}{
	author={Huisken, Gerhard},
	author={Ilmanen, Tom},
	title={The inverse mean curvature flow and the Riemannian Penrose inequality},
	journal={J. Differential Geom.},
	volume={59},
	year={2001},
	number={3},
	pages={353--437},
}
\bib{HI-2008}{article}{
	author={Huisken, Gerhard},
	author={Ilmanen, Tom},
	title={Higher regularity of the inverse mean curvature flow},
	journal={J. Differential Geom.},
	volume={80},
	year={2008},
	number={3},
	pages={433--451},
}





\bib{LWX-14}{article}{
	author={Li, Haizhong},
	author={Wei, Yong},
	author={Xiong, Changwei},
	title={A note on Weingarten hypersurfaces in the warped product manifold},
	journal={Internat. J. Math.},
	volume={25},
	year={2014},
    number={14},
	pages={1450121, 13 pp.},
}

\bib{LW17}{article}{
	author={Li, Haizhong},
	author={Wei, Yong},
	title={On inverse mean curvature flow in Schwarzschild space and Kottler space},
	journal={Calc. Var. Partial Differential Equations},
	volume={56},
	year={2017},
	number={3},
	pages={Paper No. 62, 21pp},
}

\bib{Maggi-12}{book}{
    author={Maggi, Francesco},
    title={Sets of finite perimeter and geometric variational problems},
    series={Cambridge Stud. Adv. Math.},
    volume={135},
    year={Cambridge, 2012, xx+454 pp},
    publisher={Cambridge University Press},
}

\bib{MS-16}{article}{
author={Makowski, Matthias},
        author={ Scheuer, Julian},
	title={Rigidity results, inverse curvature flows and {A}lexandrov-{F}enchel type inequalities in the sphere},
	journal={Asian J. Math.},
	volume={20},
	year={2016},
    number={5},
	pages={869--892},
}





\bib{Mc-18}{article}{
	author = {McCormick, Stephen},
	title = {On a Minkowski-like inequality for asymptotically flat static manifolds},
	journal = {Proc. Amer. Math. Soc},
	volume = {146},
	year = {2018},
	number = {9},
	pages = {4039--4046},
}

\bib{Mc05}{article}{
	author={McCoy, James A.},
	title={Mixed volume preserving curvature flows},
	journal={Calc. Var. Partial Differential Equations},
	volume={24},
	year={2005},
	number={2},
	pages={131--154},
}


\bib{Ne10}{article}{
 title={Insufficient convergence of inverse mean curvature flow on asymptotically hyperbolic manifolds},
  author={Neves, Andr{\'e}},
  journal={Journal of Differential Geometry},
  volume={84},
  number={1},
 pages={191--229},
  year={2010},
  publisher={Lehigh University}
}
\bib{Manuel}{book}{
	author={Ritor{\'e}, Manuel},
	title={Isoperimetric inequalities in Riemannian manifolds},
	volume={348},
   year={2023},
  publisher={Springer Nature},
}
  
\bib{Manuel-04}{article}{
	author={Ritor{\'e}, Manuel},
        author={Rosales, César},
	title={Existence and characterization of regions minimizing perimeter under a volume constraint inside Euclidean cones},
	journal={Trans. Amer. Math. Soc.},
	volume={356},
	year={2004},
    number={11},
	pages={4601--4622},
}

\bib{San04}{book}{
	title={Integral geometry and geometric probability},
  author={Santal{\'o}, Luis A.},
  year={2004},
  publisher={Cambridge university press},
}

\bib{Scheuer17}{article}{
	author={Scheuer, Julian},
	title={The inverse mean curvature flow in warped cylinders of non-positive radial curvature},
	journal={Adv. Math.},
	volume={306},
	year={2017},
	pages={ 1130--1163},
}
\bib{Scheuer19}{article}{
	author={Scheuer, Julian},
	title={Inverse curvature flows in Riemannian warped products},
	journal={J. Funct. Anal.},
	volume={276},
	year={2019},
	number={4},
	pages={ 1097--1144},
}
\bib{Scheuer22}{article}{
	author={Scheuer, Julian},
	title={Minkowski inequalities and constrained inverse curvature flows in warped spaces},
	journal={Adv. Calc. Var.},
	volume={15},
	year={2022},
	number={4},
	pages={735--748},
}


\bib{SZ-21}{article}{
	author = {Shi, Yuguang},
       author={Zhu, Jintian},
	title = {Regularity of inverse mean curvature flow in asymptotically hyperbolic manifolds with dimension 3},
	journal = {Sci. China Math.},
	volume = {64},
	year = {2021},
	number = {6},
	pages = {1109--1126},
}

\bib{Sol06}{article}{
	author = {Solanes, Gil},
	title = {Integral geometry and the Gauss-Bonnet theorem in constant curvature spaces},
	journal = {Trans. Amer. Math. Soc},
	volume = {358},
	year = {2006},
	number = {3},
	pages = {1105--1115},
}


 \bib{Urbas1991}{article}{
 	author={Urbas, John I. E.},
 	title={On the expansion of starshaped hypersurfaces by symmetric functions of their principal curvatures},
 	journal={Math. Z.},
 	volume={205},
 	year={1990},
 	number={3},
 	pages={355--372},
 }



\bib{Wang15}{article}{
	author={Wang, Zhuhai},
	title={A Minkowski-type inequality for hypersurfaces in the Reissner-Nordström-anti-deSitter manifold},
	journal={Ph. D. dissertation(Columbia University, 2015)},
}

\bib{Wei18}{article}{
	author={Wei, Yong},
	title={On the Minkowski-type inequality for outward minimizing hypersurfaces in Schwarzschild space},
	journal={Calc. Var. Partial Differential Equations},
	volume={56},
	year={2018},
	number={2},
	pages={Paper No. 46, 17 pp},
}


\end{thebibliography}
\end{document}